\documentclass[12pt]{amsart} 

\usepackage{doi}

\usepackage[latin9]{inputenc}
\usepackage{amssymb}
\usepackage{latexsym}
\usepackage{amsmath}
\usepackage{amsthm}
\usepackage{amsfonts}
\usepackage{tikz}

\newtheorem{theorem}{Theorem}

\newtheorem{lemma}[theorem]{Lemma}
\newtheorem{corollary}[theorem]{Corollary}
\newtheorem{proposition}[theorem]{Proposition}

\theoremstyle{definition}

\newtheorem{example}[theorem]{Example}
\newtheorem*{acknowledgements*}{Acknowledgements}
\theoremstyle{remark}
\newtheorem{remark}[theorem]{Remark}

\numberwithin{equation}{section}
\numberwithin{theorem}{section}

\newcommand{\R}{\mathbb R}

\newcommand{\C}{\mathbb C}

\title
[Fine spectra of the finite Hilbert transform]
{Fine spectra of the finite Hilbert transform \\ in function spaces}

\author[G.P.  Curbera]{Guillermo P. Curbera}
\address{Facultad de Matem\'aticas \& IMUS,
Universidad de Sevilla, 
Calle Tarfia s/n,  Sevilla 41012, Spain}
\email{curbera@us.es}

\author[S. Okada]{Susumu Okada}
\address{School of Natural Sciences (Maths/ Physics), University of Tasmania, 
Private Bag 37, Hobart, Tas. 7001, Australia}
\email{susumu.okada@utas.edu.au}

\author[W.J. Ricker]{Werner J. Ricker}
\address{Math.--Geogr. Fakult\"at, Katholische Universit\"at
Eichst\"att--Ingolstadt, D--85072 Eichst\"att, Germany}
\email{werner.ricker@ku.de}

\thanks{The first author acknowledges the support  of 
of PGC2018-096504-B-C31, FQM-262 and Feder-US-1254600 (Spain).}

\date{\today}



\begin{document}

\begin{abstract}
We investigate the spectrum and fine spectra of the finite Hilbert transform acting on rearrangement
invariant spaces over $(-1,1)$ with non-trivial Boyd indices, thereby extending
Widom's results for $L^p$ spaces. In the case when these
indices coincide, a full description of the spectrum and fine spectra is given.
\end{abstract} 

\maketitle


\section{Introduction}
\label{S1}


An important part of the theory of singular integral operators (with piecewise continuous coefficients)
acting on function spaces defined over curves is to understand their spectrum (essential, local, etc.),
which can be a rather complicated set. It depends very much on the domain space of the operator;
see \cite{bk}, \cite{k1}, \cite{k2}, and the references therein. The aim of this paper is to investigate
the spectra of one (particular, but classical) singular integral operator acting  in a class of rearrangement
invariant  spaces over  a particularly simple curve, namely the interval $(-1,1)$.

The finite Hilbert transform $T(f)$  of $f\in L^1(-1,1)$ is the principal value integral
\begin{equation}\label{T}
(T(f))(t)=\lim_{\varepsilon\to0^+} \frac{1}{\pi i}
\left(\int_{-1}^{t-\varepsilon}+\int_{t+\varepsilon}^1\right) \frac{f(x)}{x-t}\,dx, 
\end{equation}
which exists for  almost all  $t\in(-1,1)$ and is a measurable function.
A celebrated theorem of M. Riesz states, for each $1<p<\infty$, that
the Hilbert transform operator  maps $L^p(\R)$ continuously  into 
itself. We will consider  the linear operator $f\mapsto T(f)$, 
which maps $L^p(-1,1)$ continuously  into 
itself (denote this operator by $T_p$ and the space by $L^p$).
The operator $T$ has important applications in aerodynamics, 
via  the airfoil equation, \cite{cheng-rott}, 
\cite[Ch.11]{king}, \cite{reissner},  \cite{tricomi-1},
\cite{tricomi}. More recently, applications have been found 
to problems arising in image reconstruction; 
see, for example, \cite{bertola-etal}, \cite{katsevich-tovbis}, \cite{sidky-etal}. 
Our results and techniques deal with the spectral theory of $T$ beyond the $L^p$-setting. Hopefully this forms a basis for future research
concerning this classical operator to other areas of mathematics. For instance, aspects from operator theory, such as inversion theory
and Fredholm properties, and from integral representations (via vector measures), already occur in the recent articles \cite{curbera-okada-ricker-1}, 
\cite{curbera-okada-ricker-2}.

The spectrum $\sigma(T_{p})$ of $T_p$ was completely identified by Widom in 1960, 
who also gave  the decomposition of the spectrum 
into its point spectrum $\sigma_{\mathrm{pt}}(T_p)$,  continuous spectrum
$\sigma_{\mathrm{c}}(T_p)$ and  residual spectrum  $\sigma_{\mathrm{r}}(T_p)$,
\cite[\S5]{widom}; see also \cite[\S13.6]{jorgens}. Since
the operators $T_p$ and their spectra play a special role in this paper it is
worthwhile to describe Widom's results for $L^p$ in detail.

For $1<p<\infty$, consider the subset of $\C$ given by
\begin{equation}\label{rp}
\mathcal{R}_p:=\big\{\pm1\big\}\cup
\left\{\lambda\in\C: \frac{1}{2\pi} 
\left|\arg\bigg(\frac{1+\lambda}{1-\lambda}\bigg)\right|\le \left|\frac12-\frac1p\right|\right\}.
\end{equation}
Then $\mathcal{R}_p$ is the region bounded by both, the circular arc with end-points $\pm1$ 
which passes through $i\cot(\pi/p)$, together with the circular arc having end-points $\pm1$ 
which passes through $i\cot(\pi/p')$, where the conjugate index $p'$ is specified by $1/p+1/p'=1$;
see the diagram below. It is an important feature  that 
\begin{equation*}\label{rp2}
\mathcal{R}_p=\mathcal{R}_{p'}.
\end{equation*}
Note that  $\mathcal{R}_2=[-1,1]$ and, for $1<p<\infty$,
that the set $\mathcal{R}_p$ increases as $|p-2|$ increases.
Observe also that $\bigcup_{1<p<\infty} \mathcal{R}_p$ equals 
$\C\setminus\big\{(-\infty,-1)\cup(1,\infty)\big\}$.


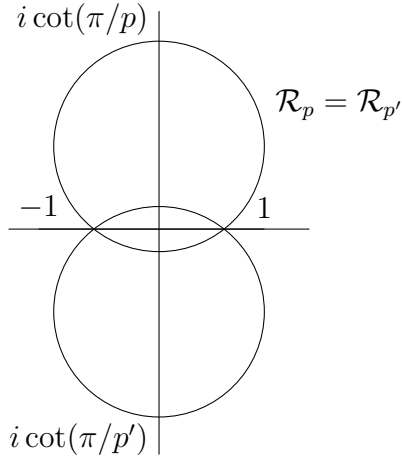
\begin{center}
\begin{tikzpicture}
 \draw[-,xshift=-0cm] (-2,0) -- coordinate (x axis mid) (2,0);
 \draw[-,xshift=-0cm] (0,-3) -- coordinate (y axis mid)(0,2.9);
\draw (0,-1.1) circle (1.4);
\draw (0,1.1) circle (1.4);
\draw (-1.6,0) node[above]{$-1$} -- (1.4,0)  node[above]{$1$};
\draw (0,2.8) node[left]{$i\cot(\pi/p)$};     
\draw (0,-2.8) node[left]{$i\cot(\pi/p')$};  
\draw (2.4,2) node[below]{$\mathcal{R}_p=\mathcal{R}_{p'}$} ;
\end{tikzpicture}
\end{center}


The following result is due to Widom; see  Remark 1(2) and Remark 2 (pp.\ 156-157) of \cite{widom}.
The interior of a set $B\subseteq \C$ is denoted by $\mathrm{int}(B)$
and its boundary by $\partial B$.

\begin{theorem}[Widom]\label{t1}
Let $1<p<\infty$. For the operator $T_p\colon L^p\to L^p$ we have 
$$
\sigma(T_p)=\mathcal{R}_p.
$$ 
Regarding the fine spectra of $T_{p}$ the following identifications hold.
\begin{itemize}
\item[(a)] Let $1<p<2$. Then $\sigma_{\emph{pt}}(T_p) =\emph{int}(\mathcal{R}_p)$,
$\sigma_{\mathrm{r}}(T_p)=\emptyset$ and $\sigma_{\mathrm{c}}(T_p)=\partial\mathcal{R}_p$.
\item[(b)] Let $p=2$. Then $\sigma_{\emph{pt}}(T_2) =\emptyset$,
 $\sigma_{\mathrm{r}}(T_2)=\emptyset$ and $\sigma_{\mathrm{c}}(T_2)=\partial\mathcal{R}_2=\mathcal{R}_2$.
\item[(c)] Let $2<p<\infty$. Then $\sigma_{\emph{pt}}(T_p) =\emptyset$,
$\sigma_{\mathrm{r}}(T_p)=\emph{int}(\mathcal{R}_p)$ and $\sigma_{\mathrm{c}}(T_p)=\partial\mathcal{R}_p$.
\end{itemize}

\end{theorem}


The following table summarizes the previous result.

\medskip
\begin{center}
\begin{tabular}{ | c | c | c | c |}
\hline
   	$\sigma(T_p)=\mathcal{R}_p$  & $\sigma_{\mathrm{pt}}(T_p) $ & 
   	$\sigma_{\mathrm{r}}(T_p) $ & $\sigma_{\mathrm{c}}(T_p) $
\\ \hline
      	&&&
\\ \hline
     	$1<p<2$ & $\mathrm{int}(\mathcal{R}_p)$ & $\emptyset$ & $\partial\mathcal{R}_p$ 
\\ \hline
     	$p=2$ & $\emptyset$ & $\emptyset$ & $\mathcal{R}_2=[-1,1]$  
\\ \hline
     	$2<p<\infty$ &$\emptyset$  & $\mathrm{int}(\mathcal{R}_p)$ & $\partial\mathcal{R}_p$
\\ \hline
\end{tabular}
\end{center}
\bigskip


As a consequence of Widom's result, the set $\mathcal{A}$ of all complex 
numbers which occur as eigenvalues for $T$ when $T$ acts  on some space $L^p$, for $1<p<\infty$, is
given by
\begin{equation}\label{AA}
\mathcal{A}=\bigcup_{1<p<\infty} \mathrm{int}(\mathcal{R}_p)=\C\setminus\big\{(-\infty,-1]\cup[1,\infty)\big\}.
\end{equation}

A result of J\"orgens from 1970 identifies the   set of all possible 
eigenfunctions for $T$ when $T$ acts over $\bigcup_{p>1}L^p$; 
see \cite[Theorem 13.9]{jorgens} for the English translation.

\begin{theorem}[J\"orgens]\label{t2}
For each  $\lambda\in\mathcal{A}$, the
corresponding eigenspace of 
$T$ is the 1-dimensional space $\langle\xi_\lambda\rangle\subseteq L^p$
spanned by the canonical eigenfunction
\begin{equation}\label{xi}
\xi_\lambda(x):=\frac{1}{(1-x^2)^{1/2}}\left(\frac{1-x}{1+x}\right)^{z(\lambda)},\quad  |x|<1,
\end{equation}
for all $1<p<\infty$ such that $\xi_\lambda\in L^p$,  where the function $z(\lambda)$ is given by
\begin{equation}\label{z}
z(\lambda):=\frac{1}{2\pi i}\log\left(\frac{1+\lambda}{1-\lambda}\right),\qquad z(0)=0.
\end{equation}
\end{theorem}


We write $\mathcal{E}=\{\xi_\lambda:\lambda\in\mathcal{A}\}$. 
Thus, the  $L^p$-theory concerning the spectrum of $T$ is well understood.

We turn our viewpoint  to the finite Hilbert transform $T$
acting on rearrangement invariant (briefly, r.i.) spaces on $(-1,1)$.
A classical result of Boyd states that the Hilbert transform acts continuously in
a r.i.\ space $X$ over $\R$ if and only if the lower and upper Boyd indices 
$\underline{\alpha}_X$ and $\overline{\alpha}_X$ of $X$ are non-trivial, that is, if
$0<\underline{\alpha}_X\le \overline{\alpha}_X<1$; see \cite[Theorem III.5.18]{bennett-sharpley}.
The same characterization is valid for the finite Hilbert transform $T$ given
by \eqref{T} when it acts in r.i.\ spaces $X$ over $(-1,1)$, 
\cite[pp.170-171]{krein-petunin-semenov},
in which case we denote $T$ by $T_X\colon X\to X$. 
This class of r.i.\ spaces is the most  adequate replacement  for the 
$L^p$-spaces when undertaking a further study of the finite Hilbert transform $T$. 
This is due to two important features: 
$T\colon X\to X$ is injective if and only if $L^{2,\infty}(-1,1)\not\subseteq X$ 
and (for $X$ separable) $T\colon X\to X$ has a 
non-dense range if and only if $X \subseteq L^{2,1}(-1,1)$, 
where $L^{2,1}(-1,1)$ and $L^{2,\infty}(-1,1)$ are 
the usual Lorentz spaces.
The class of r.i.\ spaces over $(-1,1)$ with non-trivial Boyd indices
is the one used throughout this paper.
It is closely connected to the family of $L^p$-spaces via the following result, 
\cite[Proposition 2.b.3]{lindenstrauss-tzafriri}.

\begin{lemma}\label{l1}
Let $X$ be any r.i.\ space  such that 
$0<\alpha<\underline{\alpha}_X\le \overline{\alpha}_X<\beta<1$.
Then there exist $p,q$ satisfying $1/\beta<p<q<1/\alpha$ such that
$L^q\subseteq X \subseteq L^p$ with
continuous  inclusions.
\end{lemma}

In particular, the previous result implies that

\begin{equation*}
\mathcal{A}=\Big\{\lambda\in\C: T(f)=\lambda f \text{ for some }   f\in X\setminus\{0\} 
\text{ and $X$ with } 
0<\underline{\alpha}_X\le \overline{\alpha}_X<1\Big\},
\end{equation*}
and that
\begin{equation*}
\mathcal{E}=\Big\{\xi\in X: 0<\underline{\alpha}_X\le \overline{\alpha}_X<1,  
\exists \lambda\in\mathcal{A} \text{ such that $\xi=\xi_\lambda$}\Big\}.
\end{equation*}
%


The aim of this paper is to study the spectrum $\sigma(T_X)$  of  
$T_X$ when $T_X$ acts on a r.i.\ space $X$ with non-trivial
Boyd indices, as well as its fine spectra, namely the 
point spectrum $\sigma_{\mathrm{pt}}(T_X)$,   the continuous spectrum
$\sigma_{\mathrm{c}}(T_X)$ and   the residual spectrum  $\sigma_{\mathrm{r}}(T_X)$. 
We are unaware of any advances in this direction beyond the work
of Widom and J\"orgens in the 1960--70's for the spaces $L^p$, $1<p<\infty$.

Section \ref{S2} is devoted to preliminaries concerning r.i.\ spaces, their indices, and  the finite Hilbert
transform.

In Section \ref{S3} we present the strong symmetry properties of the spectrum and fine spectra of $T_X$
(cf. Proposition \ref{p1}). In addition, we identify its point spectrum 
$\sigma_{\mathrm{pt}}(T_X)$ (cf. Proposition \ref{p2})
which turns out to be one of only three possible alternatives:
$\mathcal{R}_{p_X}\setminus\{\pm1\}$ or  $\mathrm{int}(\mathcal{R}_{p_X})$ 
or $\emptyset$, where
$p_X\in (1,\infty)$  is an index associated to $X$; see \eqref{px}.

Section \ref{S4} is devoted to identifying the  spectrum and fine spectra of $T$ when it acts in the
class of Lorentz $L^{p,r}$ spaces, for $1<p<\infty$ and $1\le r<\infty$; see Theorem \ref{t3}. 
Note that the non-separable spaces $L^{p,\infty}$ are not considered (see, in any case,
Remark \ref{r5}). This is because we mostly focus our study  on the separable r.i.\ spaces.
The reason is,  if $X$ is separable, then we are in the desirable situation
that the topological dual $X^*$ of $X$  is isometrically isomorphic to the associate space of $X$, i.e., $X^*=X'$, which is also a r.i.\ space
with non-trivial Boyd indices. Moreover, the dual operator $(T_X)^*=-T_{X'}$.

In Section \ref{S5} we exhibit further properties of the spectrum and fine spectra of $T_X$.
In Proposition \ref{p7} it is established that always $0\in\sigma(T_X)$ and shown that the fine spectrum
is strongly influenced by which part of it contains $0$. 
This leads to the somewhat unexpected situation that only three mutually distinct
cases can arise (for \textit{all} $X$), namely
\begin{equation}\label{3cases}
\sigma(T_X)=\sigma_{\mathrm{pt}}(T_X)\cup\sigma_{\mathrm{c}}(T_X);\;
\sigma(T_X)=\sigma_{\mathrm{r}}(T_X)\cup\sigma_{\mathrm{c}}(T_X);\;
\sigma(T_X)=\sigma_{\mathrm{c}}(T_X).
\end{equation}
As noted above the spaces $L^{2,\infty}$ and $L^{2,1}$
play an important role in this classification. Consequently, it turns out that 
$[-1,1]\subseteq \sigma(T_X)$; see Corollary \ref{c2}.

In Section \ref{S6} we determine the residual spectrum 
$\sigma_{\mathrm{r}}(T_X)$. In order to do so,
a further index $q_X$ associated to $X$ (see  \eqref{qx}) is required,
which allows an investigation of each of the three cases 
that occur naturally, that is,  whether $0\in \sigma_{\mathrm{pt}}(T_X)$, $0\in\sigma_{\mathrm{r}}(T_X)$
or $0\in \sigma_{\mathrm{c}}(T_X)$; see Propositions \ref{p8}, \ref{p9} and  \ref{p10}.

In the final Section \ref{S7} we present some results which provide 
a full description of the spectrum and  the fine spectra  of $T_X$.
For those r.i.\ spaces $X$ having equal Boyd indices we show that $\sigma(T_X)=\mathcal{R}_{p_X}$
and describe in detail the fine spectra; see Theorem \ref{t4}. 
For r.i.\ spaces $X$ satisfying $2>p_X=q_X$ (with $q_X$ attained) and which are 
interpolation spaces between  $L^2$ and $L^{q_X}$, it is also shown
that $\sigma(T_X)=\mathcal{R}_{p_X}$ and we describe in detail the fine spectra;
see Theorem \ref{t5}. An analogous result holds for the case  
when $p_X=q_X>2$; see Theorem \ref{t6}. 
These results are possible thanks to 
Propositions \ref{p11}, \ref{p12} and \ref{p13}, in which a suitable 
superset containing  the spectrum is identified. 
We conclude the paper with some relevant examples of families of r.i.\ spaces to 
which the full identification of the spectrum applies: various Orlicz spaces,
certain Lorentz $\Lambda$-spaces and the small Lebesgue spaces (i.e., the associate spaces
of the grand Lebesgue spaces $L^{p)}$).


The difficulty with providing a full description of the spectrum and 
the fine spectra of $T_X$, for all r.i.\ spaces
$X$ with non-trivial Boyd indices, lies with the precise identification of the continuous spectrum.
Already for the Lorentz spaces $L^{2,r}$, $1\le r<\infty$, 
to establish this identification is rather involved; see Proposition \ref{p6}. In this regard it is relevant to point out  similarities with 
the study of  the spectra of Fourier multiplier operators. 
For a Fourier multiplier operator $T$ on  $L^p(G)$ over an infinite, compact abelian group $G$,
it is known that $\sigma_{\mathrm{r}}(T)=\emptyset$. Thus,  attention was focused since long ago 
on the existence and nature of the set $\sigma_{\mathrm{c}}(T)\setminus\overline{\sigma_{\mathrm{pt}}(T)}$.
Celebrated and highly non-trivial results of Igari and Zafran gave unexpected descriptions 
of this part of the continuous spectrum of $T$; see, for example, 
\cite{zafran} and the references therein. These are similar types of difficulties that
appear with the finite Hilbert transform $T_X$, where the set 
$\overline{\sigma_{\mathrm{pt}}(T)\cup\sigma_{\mathrm{r}}(T)}$ is always known.
Indeed, this set is precisely one of $\emptyset$, $\mathcal{R}_{p_X}$ or $\mathcal{R}_{q_X}$.


\section{Preliminaries}
\label{S2}


The setting of this paper is the measure space  $(-1,1)$ 
equipped with its Borel $\sigma$-algebra $\mathcal{B}$ and  Lebesgue measure $|\cdot|$
(restricted to $\mathcal{B}$). We  denote by 
$L^0(-1,1)=L^0$ the space (of equivalence classes) of all $\mathbb{C}$-valued
measurable functions. The space $L^p(-1,1)$ is denoted simply by $L^p$, for $1\le p\le\infty$.

A \textit{rearrangement invariant} (r.i.) space $X$ on $(-1,1)$ is a
Banach space  $X\subseteq L^0$ satisfying 
that if $g^*\le f^*$ with $g\in L^0$ and $f\in X$,  
then $g\in X$ and $\|g\|_X\le\|f\|_X$. Here $f^*\colon[0,2]\to[0,\infty]$ is 
the decreasing rearrangement of $f$, that is, the right continuous inverse of its distribution function:
$\lambda\mapsto|\{t\in (-1,1):\,|f(t)|>\lambda\}|$.
Every r.i.\ space on $(-1,1)$ satisfies $L^\infty\subseteq X\subseteq L^1$, 
\cite[Corollary II.6.7]{bennett-sharpley}.

The associate space $X'$  of $X$ consists  of all
functions $g\in L^0$ satisfying $\int_{-1}^1|fg|<\infty$, for every
$f\in X$. It is equipped with the norm
$\|g\|_{X'}:=\sup\{|\int_{-1}^1fg|:\|f\|_X\le1\}$. 
The space $X'$ is isometrically isomorphic to a 
closed subspace of the Banach space dual $X^*$ of $X$. 
The associate space $X'$  is again a r.i.\ space.
The second  associate space $X''$ 
of $X$ is defined as $X''=(X')'$. 
Moreover, if $f\in X$ and $g\in X'$, then $fg\in L^1$ and
$\|fg\|_{L^1}\le \|f\|_X \|g\|_{X'}$, that is, H\"older's inequality is available. 
The fundamental function $\varphi_X$ of $X$ is defined by  $\varphi_X(t):=\|\chi_{A}\|_X$ 
for any set $A\in\mathcal{B}$ 
with $|A|=t$, for $t\in[0,2]$.
The norm in $X$ is absolutely continuous if,
for every $f\in X$, we have $\|f\chi_A\|_X\to0$ whenever $|A|\to0$.
The space $X$ satisfies the Fatou property  if, whenever  $\{f_n\}_{n=1}^\infty\subseteq X$ satisfies
$0\le f_n\le f_{n+1}\uparrow f$ a.e.\ with $\sup_n\|f_n\|_X<\infty$,
then $f\in X$ and $\|f_n\|_X\to\|f\|_X$.   
In this paper all r.i.\   spaces, as in \cite{bennett-sharpley},    
satisfy the Fatou property. In this case $X''=X$ and hence,
$f\in X$ if and only if $\int_{-1}^1|fg|<\infty$, for every $g\in X'$.
If $X$ is separable, then $X'=X^*$,
\cite[Corollaries I.3.4 and I.5.6]{bennett-sharpley}. .

The family of r.i.\ spaces includes many classical spaces, such as  the Lorentz $L^{p,q}$ spaces, 
\cite[Definition IV.4.1]{bennett-sharpley}, Orlicz $L^\Phi$ spaces 
\cite[\S4.8]{bennett-sharpley}, Marcinkiewicz $M_\varphi$ spaces, 
\cite[Definition II.5.7]{bennett-sharpley}, Lorentz $\Lambda_\varphi$ spaces,
\cite[Definition II.5.12]{bennett-sharpley},
and the Zygmund $L^p(\text{log L})^\alpha$ spaces, 
\cite[Definition IV.6.11]{bennett-sharpley}.

The dilation operator $E_t$ for $t>0$ is defined, for 
each $f\in X$, by $E_t(f)(s):=f(st)$ for $-1\le st\le1$ and zero in 
other cases. The operator $E_t\colon X\to X$  is bounded 
with $\|E_t\|_{X\to X}\le \max\{t,1\}$. The \textit{lower} and \textit{upper 
Boyd indices} of  $X$ are defined, respectively, by
\begin{equation*}
\underline{\alpha}_X\,:=\,\sup_{0<t<1}\frac{\log \|E_{1/t}\|_{X\to X}}{\log t}
\;\;\mbox{and}\;\;
\overline{\alpha}_X\,:=\,\inf_{1<t<\infty}\frac{\log \|E_{1/t}\|_{X\to X}}{\log t} ,
\end{equation*}
\cite[Definition III.5.12]{bennett-sharpley}. 
They satisfy $0\le\underline{\alpha}_X\le \overline{\alpha}_X\le1$.
Note that $\underline{\alpha}_{L^p}= \overline{\alpha}_{L^p}=1/p$
for all $1\le p<\infty$.
The lower and upper fundamental indices, 
$\underline{\beta}_X$ and $\overline{\beta}_X$, respectively, are  
defined by
\begin{equation*}
\underline{\beta}_X:= \sup_{0<t<1}\frac{\log M_{\varphi_X}(t)}{\log t}
\;\;\mbox{and}\;\;
\overline{\beta}_X:= \inf_{t>1}\frac{\log M_{\varphi_X}(t)}{\log t},
\end{equation*}
\cite[pp. 177-178]{bennett-sharpley}, where
$$
M_{\varphi_X}(t):=\sup_{0<st\le1}\frac{\varphi_X(st)}{\varphi_X(s)}.
$$
The following relation between these indices holds:
$$
0\le\underline{\alpha}_X\le \underline{\beta}_X\le
\overline{\beta}_X\le \overline{\alpha}_X\le 1.
$$
Those r.i.\ spaces $X$ for which $\underline{\alpha}_X= \underline{\beta}_X$
and $\overline{\beta}_X= \overline{\alpha}_X$ are said to be of fundamental type,
\cite{feher}, \cite{maligranda}.

An important role for the finite Hilbert transform is played by the  Marcinkiewicz  space 
$L^{2,\infty}(-1,1)=L^{2,\infty}$, also known as  weak-$L^2$, \cite[Definition IV.4.1]{bennett-sharpley}. 
It consists of those functions $f\in L^0$  satisfying
\begin{equation*}\label{L2oo}
f^*(t)\le \frac{M}{t^{1/2}},\quad 0<t\le2,
\end{equation*}
for some constant $M>0$. Consider the function $1/\sqrt{1-x^2}$ on $(-1,1)$. 
Since its decreasing rearrangement  $(1/\sqrt{1-x^2})^*$
is the function $t\mapsto 2/t^{1/2}$,  it follows that $1/\sqrt{1-x^2}$ belongs to
$L^{2,\infty}$. Actually, for any r.i.\ space $X$  it is the case that 
$1/\sqrt{1-x^2}\in X$ if and only if $L^{2,\infty}\subseteq X$. 
To see this, suppose that $1/\sqrt{1-x^2}\in X$. Fix $f\in L^{2,\infty}$.
Then, for some $M>0$, we have
$f^*(t)\le Mt^{-1/2}=(M/2)(1/\sqrt{1-x^2})^*(t)$ for $0<t<2$.
Since $1/\sqrt{1-x^2}\in X$, it follows from the definition of $X$ being r.i. that $f\in X$. 
Accordingly, $L^{2,\infty}\subseteq X$. 
The converse is immediate because $1/\sqrt{1-x^2}\in L^{2,\infty}$.
Consequently, $L^{2,\infty}$ is the \textit{smallest} r.i.\ space  
which contains  $1/\sqrt{1-x^2}$. Note that  
$\underline{\alpha}_{L^{2,\infty}}= \overline{\alpha}_{L^{2,\infty}}=1/2$.

Regarding
the finite Hilbert transform  acting on  any r.i.\ space $X$ with non-trivial Boyd indices,
it follows from the Parseval formula for $T_X$, 
(cf.\  \cite[Proposition 3.1(b)]{curbera-okada-ricker-1}), that the restriction
of the dual operator $(T_X)^*\colon X^*\to X^*$ of $T_X$ to the  associate
space $X'$ is precisely $-T_{X'}\colon X'\to X'$. Recalling that 
$\underline{\alpha}_{X'}= 1-\overline{\alpha}_X$ and $\overline{\alpha}_{X'}=
1-\underline{\alpha}_X$, \cite[Proposition III.5.13]{bennett-sharpley},
we see that the r.i.\  space $X'$ also has non-trivial Boyd indices and so
$T_{X'}$ is well defined.

Regarding the operator $T$ itself, the definition \eqref{T}  is used by  Widom, \cite{widom}, 
and J\"orgens, \cite[\S13.6]{jorgens}. 
Other  definitions of the finite 
Hilbert transform also appear in the literature, differing
from the one above by a multiplicative constant; e.g., Tricomi, \cite[\S4.3]{tricomi},
uses $iT$ whereas King, \cite[Ch.11]{king}, and Rooney, \cite{rooney}, use $-T$.

Regarding notation, $A\asymp B$ means that there exists absolute constants $c,C>0$
such that $cA\le B\le CA$.

For all of the above (and further) facts on r.i.\  spaces see \cite{bennett-sharpley}, 
\cite{lindenstrauss-tzafriri}, for example.


\section{General properties of the spectrum of $T_X$. Point spectrum.}
\label{S3}


Let $X$ be any r.i.\ space on $(-1,1)$  with non-trivial Boyd indices and $T_X\colon X\to X$.
The  spectrum and the fine spectra of $T_X$ have strong symmetry properties.

\begin{proposition}\label{p1}
Let $X$ be any   r.i.\ space  on $(-1,1)$  such that   $0<\underline{\alpha}_X\le \overline{\alpha}_X<1$. 
\begin{itemize}
\item [(a)] Each of the spectra $\sigma(T_X)$, $\sigma_{\mathrm{pt}}(T_X)$, $\sigma_{\mathrm{c}}(T_X)$ and $\sigma_{\mathrm{r}}(T_X)$  is symmetric with respect to both  the real axis and the imaginary axis in $\C$.
In particular, these spectra are also symmetric with respect to reflection through 0.

\item [(b)] The set $\sigma_{\mathrm{pt}}(T_{X})$ is $\R$-balanced, that is, $\alpha\lambda\in \sigma_{\mathrm{pt}}(T_{X})$ for each 
$\lambda\in \sigma_{\mathrm{pt}}(T_{X})$ and every $\alpha\in\R$ satisfying $|\alpha|\le1$.

\end{itemize}
\end{proposition}

\begin{proof}
(a) Define  the reflection operator $S$ by $S(f)(x):=f(-x)$, for $|x|<1$. It is
routine to verify from the definition of 
the finite Hilbert transform $T$ that
$$
\overline{T(f)}=-T(\overline{f}) \quad\text{and that}\quad
S\big(T(f)\big)=-T\big(S(f)\big),\quad f\in L^1,
$$
where $\overline{f}$ denotes the complex conjugate of $f$.
These identities can be used to establish the following facts.
\begin{itemize}
\item[(i)] $(\lambda I -T)$ is injective $\iff (\lambda I +T)$  is injective 
$ \iff (\overline{\lambda} I -T)$   is injective.
\item[(ii)]   $(\overline{\lambda} I-T)(X)= \overline{S(\lambda I-T)(X)}
:=\{\overline{g}:g\in S(\lambda I-T)(X)\}$.
\item[(iii)]  $S(\lambda I+T)(X)= (\lambda I-T)(X)$.
\end{itemize}

Facts (i)-(iii) can be used to verify that $B=-B$  and $\overline{B}=-B$,
where $B$ denotes any one of the spectra  
$\sigma(T_X)$, $\sigma_{\mathrm{pt}}(T_X)$, $\sigma_{\mathrm{c}}(T_X)$ or 
$\sigma_{\mathrm{r}}(T_X)$.

(b) Let $\lambda\in\sigma_{\mathrm{pt}}(T_{X})$. Then 
the corresponding canonical eigenfunction $\xi_\lambda$ 
given by \eqref{xi} belongs to $X$. 
Fix $\alpha\in\R$ with $|\alpha|\le1$. 
It suffices to  prove that $\xi_{\alpha \lambda}\in X$, from which
it follows that $\alpha\lambda\in\sigma_{\mathrm{pt}}(T_{X})$.

Note that the M\"obius transformation 
\begin{equation}\label{u}
u(\lambda):=\frac{1+\lambda}{1-\lambda}
\end{equation}
maps 
the set $\mathcal{A}$ of all eigenvalues given in \eqref{AA} onto the set $\Omega:=\C\setminus(-\infty,0]$.
In $\Omega$ the argument used for complex numbers can be fixed to lie in $(-\pi,\pi)$. 
This implies, for the function $z(\cdot)$ in \eqref{z}, that its real part 
is defined by
\begin{equation}\label{re-z}
\Re(z(\lambda))=\frac{1}{2\pi}\arg\left(\frac{1+\lambda}{1-\lambda}\right)
=\frac{1}{2\pi}\arg(u(\lambda))
\in\Big(\frac{-1}{2},\frac12\Big).
\end{equation}

From the definition \eqref{xi} of $\xi_\lambda$, the definition  \eqref{z} of $z(\lambda)$, and
\eqref{u}  with \eqref{re-z}, it follows that 
\begin{equation}\label{mod-xi}
\left| \xi_{\lambda}(x)\right|= 
\frac{1}{(1-x^2)^{1/2}}\left(\frac{1-x}{1+x}\right)^{\Re(z(\lambda))}
=
\frac{1}{(1-x^2)^{1/2}}\left(\frac{1-x}{1+x}\right)^{\frac{1}{2\pi}\arg(u(\lambda))}
\end{equation}
and 
$$
\left| \xi_{\alpha\lambda}(x)\right|= 
\frac{1}{(1-x^2)^{1/2}}\left(\frac{1-x}{1+x}\right)^{\Re(z(\alpha\lambda))}
=
\frac{1}{(1-x^2)^{1/2}}\left(\frac{1-x}{1+x}\right)^{\frac{1}{2\pi}\arg(u(\alpha\lambda))}.
$$

Recall   that $\lambda\in\mathcal{A}$ and $u(\mathcal{A})=\Omega=\C\setminus(-\infty,0]$. 
By part (a), also $-\lambda\in\sigma_{\mathrm{pt}}(T_{X})$.
The point $\alpha\lambda$ belongs to the complex segment 
$[\lambda,-\lambda]\subseteq \mathcal{A}$ which joins  $\lambda$ with $-\lambda$ and contains zero.
The image of $[\lambda,-\lambda]$ under the map $u(\cdot)$  
is the arc $\Gamma$ of the circle in $\Omega$
which joins $u(\lambda)$ and $u(-\lambda)$ and contains $u(0)=1$. 
Hence, $u(\alpha\lambda)\in\Gamma$.
Note that
$u(-\lambda)=1/u(\lambda)$. Since the argument function in $\Omega$ takes its values 
in $(-\pi,\pi)$, we have
$\arg(u(-\lambda))=-\arg(u(\lambda))$.  Hence, since $u(\alpha\lambda)\in\Gamma$,
it follows that $|\arg(u(\alpha\lambda))|\le|\arg(u(\lambda))|$.

Let $\widetilde X$ be a r.i.\  space on $(0,2)$ given by the Luxemburg representation for 
$X$, \cite[Theorem II.4.10]{bennett-sharpley}. Then $\xi_\lambda \in X$ precisely when 
its decreasing rearrangement
$\xi_\lambda^* \in \widetilde X$. It follows from \eqref{mod-xi} that
$$
\big|\xi_\lambda(x)\big|
=
\frac{(1-x)^{\frac{1}{2\pi}\arg\left(\frac{1+\lambda}{1-\lambda}\right)-\frac12}}
{(1+x)^{\frac{1}{2\pi}\arg\left(\frac{1+\lambda}{1-\lambda}\right)+\frac12}},\quad |x|<1.
$$
This identity implies  that
\begin{equation}\label{xi*}
\xi_\lambda^* (t) \asymp 
t^{-\left(\frac12+\frac{1}{2\pi}\left|\arg\left(\frac{1+\lambda}{1-\lambda}\right)\right|\right)},
\quad 0<t<2.
\end{equation}

Consider  the decreasing rearrangements  $\xi_\lambda^*$ and $(\xi_{\alpha\lambda})^*$ 
of the functions $\xi_\lambda$ and $\xi_{\alpha\lambda}$, respectively. 
Since   $|\arg(u(\alpha\lambda))|\le|\arg(u(\lambda))|$, it follows
that 
$$
(\xi_{\alpha\lambda})^*(t)\le(\xi_\lambda)^*(t),\quad 0<t<1.
$$ 
But 
$\xi_{\lambda}\in X$ and $X$ is r.i., from which it follows that $\xi_{\alpha\lambda}\in X$.
\end{proof}


We now address the point spectrum 
of $T_X$. Given  any r.i.\ space $X$ on $(-1,1)$ with non-trivial Boyd indices,
define $p_X\in(1,\infty)$ by
\begin{equation}\label{px}
p_X:=\inf\Big\{p\in(1,\infty): |x|^{-1/p}\in X\Big\}=\inf\Big\{p\in(1,\infty): 
L^{p,\infty}\subseteq X\Big\},
\end{equation}
where we have used the fact that $|x|^{-1/p}\in X$ if and only if
$L^{p,\infty}\subseteq X$.
Choose $1<r<\infty$ such that $X\subseteq L^r$ (cf.\ Lemma \ref{l1}). Then
$r=p_{L^r}\le p_X$ shows that necessarily $p_X>1$.
The index  $p_X$ can be  attained or not, depending on the space $X$, 
a fact which will be relevant for the study of the spectrum of $T$.
For $X=L^{p,r}$ with $1<p<\infty$ and $1\le r<\infty$, we have that $p_X=p$ is not attained,
whereas for $X=L^{p,\infty}$ we have that $p_X=p$  is attained . Note that
\begin{equation}\label{px2}
p_X=\inf\Big\{p\in(1,\infty): L^{p}\subseteq X\Big\}.
\end{equation}
The right-side of \eqref{px2}, while  giving the value of $p_X$, may be attained or not
attained in different spaces than the original definition given in \eqref{px}.


For each $\lambda\in\mathcal{A}$, define the number $\gamma_\lambda$  by
\begin{equation}\label{gl}
\frac{1}{\gamma_\lambda} := \frac12+\frac{1}{2\pi}
\left|\arg\bigg(\frac{\lambda+1}{\lambda-1}\bigg)\right|.
\end{equation}
Note, from the determination of the complex argument in \eqref{re-z}, that  $1<\gamma_\lambda\le2$.


\begin{lemma}\label{l2}
Let $X$ be any   r.i.\ space  on $(-1,1)$  such that   $0<\underline{\alpha}_X\le \overline{\alpha}_X<1$.
Let $\lambda\in\mathcal{A}$ with $\xi_\lambda\in\mathcal{E}$ the 
corresponding canonical eigenfunction. 
\begin{itemize}
\item[(a)] Suppose that $p_X$ is  attained. Then 
$$
\xi_\lambda \in X  \iff p_X\le \gamma_\lambda.
$$
\item[(b)] Suppose that $p_X$ is not attained. Then
$$
\xi_\lambda \in X  \iff p_X<\gamma_\lambda.
$$
\end{itemize}
\end{lemma}

\begin{proof}
Let $\lambda\in\mathcal{A}$ with $\xi_\lambda\in\mathcal{E}$. For  the decreasing rearrangement $\xi_\lambda^*$ of $\xi_\lambda$ we recall, from \eqref{xi*}, that  
$$
\xi_\lambda^* (t) \asymp 
t^{-\left(\frac12+\frac{1}{2\pi}\left|\arg\left(\frac{1+\lambda}{1-\lambda}\right)\right|\right)},
\quad 0<t<2.
$$
In view of the definition \eqref{gl} it follows that 
$$
\xi_\lambda^* (t) \asymp t^{-1/\gamma_\lambda},
\quad 0<t<2.
$$
Hence, $\xi_\lambda\in L^{(1/\gamma_\lambda),\infty}$. 
So, $\xi_\lambda \in X$ precisely when $|x|^{-1/\gamma_\lambda}\in X$.
In view of the definition of $p_X$ given in \eqref{px}, this condition is 
precisely $p_X\le \gamma_\lambda$ or  $p_X< \gamma_\lambda$, 
depending on which case (a) or (b) we are considering. 
\end{proof}

The point spectrum of $T_X$ can now be completely  identified.

\begin{proposition}\label{p2}
Let $X$ be any   r.i.\ space  on $(-1,1)$  such that   $0<\underline{\alpha}_X\le \overline{\alpha}_X<1$.
\begin{itemize}

\item [(a)] Let $p_X>2$ (attained or not) or, let $p_X=2$ with $p_X$  not attained.
Then $\sigma_{\mathrm{pt}}(T_X)=\emptyset$.

\item [(b)]   Let $p_X\le2$ with $p_X$  attained. Then 
$\sigma_{\mathrm{pt}}(T_X)=\mathcal{R}_{p_X}\setminus\{\pm1\}$.

\item [(c)] Let $p_X<2$ with $p_X$  not attained. Then
$\sigma_{\mathrm{pt}}(T_X)=\mathrm{int}(\mathcal{R}_{p_X})$.

\end{itemize}
\end{proposition}

\begin{proof} 
From Theorems \ref{t1} and \ref{t2} and Lemma \ref{l2} we have 
$T(f)=\lambda f$, for $f\in X\setminus\{0\}$   and $\lambda\in\C$
precisely when  $\lambda\in\mathcal{A}$ and   
$f\in \langle \xi_\lambda\rangle$.
So, the point spectrum $\sigma_{\mathrm{pt}}(T_X)$ 
consists of  those $\lambda\in\C$ for which $\xi_\lambda\in X$, that is,
$\sigma_{\mathrm{pt}}(T_X)=\{\lambda\in\C: \xi_\lambda\in \mathcal{E}\cap X\}$.

(a) Since $\gamma_\lambda\le2$ for all $\lambda\in\mathcal{A}$, 
if $p_X>2$, then there is no 
$\lambda\in\mathcal{A}$ for which $p_X\le\gamma_\lambda$. 
In the case $p_X=2$ with
$p_X$ not attained, the condition in Lemma \ref{l2}(b) for $\lambda$ 
to satisfy $\xi_\lambda\in X$
is that $p_X<\gamma_\lambda$. Since $2=p_X$ and $\gamma_\lambda\le2$, 
there is
no $\lambda$ which is an eigenvalue for $T_X$.

(b) Let $\lambda\in\C$. Since  $1/p_X\ge 1/2$,  Lemma \ref{l2} implies that $\xi_\lambda \in X$ 
(i.e., $\lambda\in\sigma_{\mathrm{pt}}(T_X)$) if and only if 
$$
\frac12+\frac{1}{2\pi} \left|\arg\bigg(\frac{\lambda+1}{\lambda-1}\bigg)\right|\le\frac{1}{p_X}
\iff 
\frac{1}{2\pi} \left|\arg\bigg(\frac{\lambda+1}{\lambda-1}\bigg)\right|\le
\left|\frac12-\frac{1}{p_X}\right|.
$$
According to \eqref{rp} this occurs  precisely when $\lambda\in \mathcal{R}_{p_X}\setminus\{\pm1\}$.

(c) Let $\lambda\in\C$. Arguing  as in (b), we arrive at $\xi_\lambda \in X$ 
(i.e., $\lambda\in\sigma_{\mathrm{pt}}(T_X)$) if and only if 
$$ 
\frac{1}{2\pi} \left|\arg\bigg(\frac{\lambda+1}{\lambda-1}\bigg)\right|<\left|\frac12-\frac{1}{p_X}\right|.
$$
In view of \eqref{rp}, this condition corresponds to $\lambda\in \mathrm{int}(\mathcal{R}_{p_X})$.
\end{proof}

\begin{remark}\label{r1}
(a) The identity
$$
\sigma_{\mathrm{pt}}(T_X)=\Big\{\lambda\in\C: \xi_\lambda\in\mathcal{E}\cap X\Big\}
$$
implies that the 
point spectrum is \textit{monotone with respect  to containment of spaces}. That is, if
$X\subseteq Y$ are r.i.\ spaces over $(-1,1)$, both with non-trivial Boyd indices, then  
$\sigma_{\mathrm{pt}}(T_X)\subseteq \sigma_{\mathrm{pt}}(T_Y)$.

(b) Let  $\lambda\in\sigma_{\mathrm{pt}}(T_X)$. Then
 $\mathrm{Ker}(\lambda I-T_X)=\langle\xi_\lambda\rangle$ is 1-dimensional.
For $X=L^p$ with $1<p<2$ this was already observed in \cite[Theorem 13.9]{jorgens} 
and  \cite[\S4.3]{tricomi}.

(c) It was noted after \eqref{px} that $p_X=p$ is attained for $X=L^{p,\infty}$,
$1<p<\infty$. Hence, Proposition \ref{p2}(b) implies that 
$\sigma_{\mathrm{pt}}(L^{p,\infty})=\mathcal{R}_{p}\setminus\{\pm1\}$.
In particular, $\sigma_{\mathrm{pt}}(L^{2,\infty})=(-1,1)$.

\end{remark}


\section{The spectrum of $T$  in the  Lorentz $L^{p,r}$ spaces}
\label{S4}


The results on the spectra of $T$ when acting on the $L^p$-spaces  and those concerning  the 
corresponding eigenfunctions 
of $T$, together with the general properties of the spectra of $T_X$  established 
in the previous section, allow us to study the spectra of $T$ when acting in the    
family of Lorentz $L^{p,r}$ spaces, for $1<p<\infty$ and $1\le r< \infty$.   
Denote the finite Hilbert transform $T\colon L^{p,r}\to L^{p,r}$ simply by $T_{p,r}$.


The following four  propositions consider  different situations 
that arise when identifying  the spectrum and
fine spectra  of $T_{p,r}$ for various choices of the  values of $1<p<\infty$  and $1\le r< \infty$.

\begin{proposition}\label{p3}
Let  $1<p<\infty$.    Then
$\sigma(T_{p,r})\subseteq \mathcal{R}_p$ for every $1\le r< \infty$.
\end{proposition}

\begin{proof} 
Let  $1\le r< \infty$.
Fix $\lambda\notin\mathcal{R}_p$. Select $\alpha\in(1,\infty)$ such that
$\alpha$ and  its conjugate index $\alpha'$ satisfy   $1<\alpha<p<\alpha'<\infty$  and 
$\lambda\notin\mathcal{R}_\alpha=\mathcal{R}_{\alpha'}$. 
Then $L^{\alpha'}\subseteq L^{p,r}\subseteq L^\alpha$ and Theorem \ref{t1}
implies that  the resolvent operators 
$R(\lambda,T_\alpha):=(\lambda I -T)^{-1}\colon L^\alpha\to L^\alpha$
and $R(\lambda,T_{\alpha'}):=(\lambda I -T)^{-1}\colon L^{\alpha'}\to L^{\alpha'}$
are continuous linear bijections.
Boyd's interpolation theorem, \cite[Theorem 2.b.11]{lindenstrauss-tzafriri}, allows us to conclude
that $(\lambda I -T)^{-1}\colon L^{p,r}\to L^{p,r}$ is continuous; 
recall that
 $\underline{\alpha}_{L^{p,r}}=\overline{\alpha}_{L^{p,r}}=1/p$,
 \cite[Theorem IV.4.3]{bennett-sharpley}.
Let $S$ denote this continuous linear operator $(\lambda I -T)^{-1}\colon L^{p,r}\to L^{p,r}$. 
Then $S\circ (\lambda I -T_{p,r})=I$ on $L^{\alpha'}$ and $(\lambda I -T_{p,r})\circ S=I$ on $L^{\alpha'}$.
Since $L^{\alpha'}$ is dense in $L^{p,r}$, it follows that $S=(\lambda I -T_{p,r})^{-1}$ on 
$L^{p,r}$. Hence, $\lambda\notin\sigma(T_{p,r})$. Consequently, $\sigma(T_{p,r})\subseteq\mathcal{R}_p$.
\end{proof}


For $X=L^{p,r}$ with $1<p<\infty$ it follows from the definition of $p_X$ in \eqref{px} that $p_X=p$, 
which is  not attained for every $1\le r<\infty$. Hence, the following result 
is an immediate consequence of Proposition \ref{p2}.

\begin{proposition}\label{p4}
Let  $1<p<\infty$   and $1\le r< \infty$.  
\begin{itemize}
\item[(a)] For $1<p<2$ and $1\le r<\infty$, we have that
	$\sigma_{\mathrm{pt}}(T_{p,r})=\mathrm{int}(\mathcal{R}_p)$.
\item[(b)] For $p=2$ and $1\le r<\infty$, we have that $\sigma_{\mathrm{pt}}(T_{2,r})=\emptyset$.
\item[(c)] For $2<p<\infty$ and $1\le r< \infty$, we have that $\sigma_{\mathrm{pt}}(T_{p,r})=\emptyset$.
\end{itemize}
\end{proposition}


\begin{proposition}\label{p5}
Let  $1<p<\infty$  with $p\not=2$.  Then
$\mathcal{R}_p \subseteq \sigma(T_{p,r})$ 
for every  $1\le r< \infty$.
\end{proposition}

\begin{proof} 
Consider  first $1<p<2$.   Proposition \ref{p4} implies that
$\sigma_{\mathrm{pt}}(T_{p,r})=\mathrm{int}(\mathcal{R}_p)$, 
from which it follows that
$$
\mathcal{R}_p=\overline{\mathrm{int}(\mathcal{R}_p)}=\overline{\sigma_{\mathrm{pt}}(T_{p,r})}
\subseteq \sigma(T_{p,r}).
$$

For the case $2<p<\infty$ we have that $1<p'<2$. Fix $1<r<\infty$. Arguing  as
in the previous paragraph (with $p'$ and $r'$ in place of $p$ and $r$) gives
$$
\mathcal{R}_{p'}=\overline{\mathrm{int}(\mathcal{R}_{p'})}=\overline{\sigma_{\mathrm{pt}}(T_{p',r'})}
\subseteq \sigma(T_{p',r'}).
$$ 
Recall that  $\mathcal{R}_{p'}=\mathcal{R}_{p}$ and,
by the separability of $L^{p,r}$, that 
$(L^{p,r})^*=(L^{p,r})'=L^{p',r'}$ (see \cite[Theorem IV.4.7]{bennett-sharpley})
and hence,
$(T_{p,r})^*=-T_{p',r'}$.
Due to the symmetry of the spectrum 
(cf. Proposition \ref{p1}), it follows that $\sigma(T_{p',r'})=\sigma(-T_{p',r'})=\sigma((T_{p,r})^*)
=\sigma(T_{p,r})$, from which  we can deduce that 
$\mathcal{R}_p\subseteq \sigma(T_{p,r})$.

For $2<p<\infty$ and $r=1$
a modification of the previous  argument is needed.
The separability of $L^{p,1}$, again together with 
\cite[Theorem IV.4.7]{bennett-sharpley}, yields 
$(L^{p,1})^*=L^{p',\infty}$ and $(T_{p,1})^*=-T_{p',\infty}$.
Fix any $1<s<\infty$. Since $L^{p',s'}\subseteq L^{p',\infty}$, 
the monotonicity of the point spectrum observed in Remark \ref{r1}(a)
implies that $\sigma_{\mathrm{pt}}(T_{p',s'})\subseteq \sigma_{\mathrm{pt}}(T_{p',\infty})$.
By the argument of the previous paragraph (with $s$ in place of $r$)
we have that
$$
\mathcal{R}_{p'}=\overline{\mathrm{int}(\mathcal{R}_{p'})}=\overline{\sigma_{\mathrm{pt}}(T_{p',s'})}
\subseteq \overline{\sigma_{\mathrm{pt}}(T_{p',\infty})}\subseteq \sigma(T_{p',\infty}).
$$
Since $\mathcal{R}_{p'}=\mathcal{R}_{p}$
and, by Proposition \ref{p1}, also  $\sigma(T_{p',\infty})=\sigma(-T_{p',\infty})=\sigma((T_{p,1})^*)
=\sigma(T_{p,1})$, it follows that  
$\mathcal{R}_p\subseteq \sigma(T_{p,1})$. 
\end{proof}


\begin{proposition}\label{p6}
Let  $p=2$.  Then
$\mathcal{R}_2=[-1,1] \subseteq \sigma(T_{2,r})$
for every  $1\le r<\infty$. 
\end{proposition}

\begin{proof} 
Let $1\le r<\infty$. It suffices to show that $(-1,1)\subseteq \sigma(T_{2,r})$. For this  we will  prove,  for every
$\lambda\in(-1,1)$, that the operator $\lambda I- T_{2,r}\colon L^{2,r}\to L^{2,r}$ is not surjective.
Fix $\lambda\in(-1,1)$ and consider the corresponding eigenfunction $\xi_\lambda$ given in \eqref{xi}.
Define
$$
g_\lambda(x):=w(x) \xi_\lambda(x)= \left(\frac{1-x}{1+x}\right)^{z(\lambda)},
\quad |x|<1,
$$
with $w(x):=(1-x^2)^{1/2}$ for $|x|<1$ and recall (cf. \eqref{z})  that
$$
z(\lambda)=\frac{1}{2\pi i}\log\left(\frac{1+\lambda}{1-\lambda}\right),\qquad z(0)=0.
$$
Since $\lambda\in(-1,1)$, we have that $(1+\lambda)/(1-\lambda)\in(0,\infty)$. Thus, $z(\lambda)$ is purely 
imaginary and so $|g_\lambda(x)|=1$ for  $|x|<1$. Hence, $g_\lambda\in L^{2,r}$.

We will show next that there exists no function $h\in L^{2,r}$ satisfying $(\lambda I-T_{2,r})(h)=g_\lambda$. This will then establish the proposition.

Assume, on the contrary, that  there does exist $h\in L^{2,r}$ such that $(\lambda I-T_{2,r})(h)=g_\lambda$.
Choose $1<p<2$. Then the operator  $\lambda I -T_p\colon L^p\to L^p$ 
is an $L^p$-valued extension of $\lambda I -T_{2,r}\colon L^{2,r}\to L^{2,r}$ 
since $L^{2,r}\subseteq L^p$. By a  result of J\"orgens,
$\lambda I -T_p\colon L^p\to L^p$  
has a pseudo-inverse $\widehat R(\lambda, T_p)\colon L^p\to L^p$, namely
$$
\widehat R(\lambda, T_p)(f):=
\frac{1}{\lambda^2-1} \Big(\lambda f+ \xi_\lambda T\big(f/\xi_\lambda\big)\Big),
\quad f\in L^p,
$$
which satisfies $(\lambda I -T_p)\widehat R(\lambda, T_p)=I$ on $L^p$,
\cite[Theorem  13.9, (13.67)]{jorgens}.

For the choice $f:=g_\lambda=w \xi_\lambda$ we have that
\begin{equation*}
\widehat R(\lambda, T_p)(g_\lambda)=
\frac{1}{\lambda^2-1} \Big(\lambda g_\lambda+ \frac{g_\lambda}{w} T(w)\Big).
\end{equation*}
By applying the identity $T(1/w)=0$ (see, \cite[\S4.3, (7)]{widom}), let us verify $T(w)(x)=ix$ for
$|x|<1$ as done in \cite[\S11.40, (11.56), p.\ 530]{king}. Namely,
\begin{align*}
T(w)(x) &=T\bigg(\frac{w^2}{w}\bigg)(x)=
\lim_{\varepsilon\to0^+} \frac{1}{\pi i}
\left(\int_{-1}^{x-\varepsilon}+\int_{x+\varepsilon}^1\right) 
\frac{(x^2-t^2)^{1/2}}{(t-x)(1-t^2)^{1/2}}\,dt
\\ &=
- \frac{1}{\pi i} \int_{-1}^1
\frac{(x+t)^{1/2}}{(1-t^2)^{1/2}}\,dt
= ix. 
\end{align*}
Hence, it follows that
\begin{equation}\label{eq-1}
\widehat R(\lambda, T_p)(g_\lambda)=
\frac{1}{\lambda^2-1} \Big(\lambda g_\lambda+ \frac{i x g_\lambda}{w} \Big).
\end{equation}

Note that a function $f\in L^{p}$ satisfies $(\lambda I-T_{p})(f)=g_\lambda$
precisely when there exists a constant $C\in\C$ such that
\begin{equation}\label{eq-2}
f=\frac{C g_\lambda}{w} + \widehat R(\lambda, T_p)(g_\lambda),
\end{equation}
which follows from the fact that $\mathrm{Ker}(\lambda I -T_{p})=
\langle \xi_\lambda\rangle$; see Remark \ref{r1}(b).

Since $h\in L^{2,r}\subseteq L^p$ satisfies $(\lambda I-T_{2,r})(h)=g_\lambda$
with $g_\lambda\in L^{p}$, \eqref{eq-1} and \eqref{eq-2} imply, for some constant $C\in\C$, that 
$$
h= C \frac{g_\lambda}{w}+
\frac{1}{\lambda^2-1} \Big(\lambda g_\lambda+ \frac{i x g_\lambda}{w} \Big),
$$
that is,
$$
h=\frac{\lambda }{\lambda^2-1} g_\lambda+ \Big(\frac{i x}{\lambda^2-1}  +C\Big)\frac{g_\lambda}{w}.
$$
In view of the fact that $|g_\lambda|=1$ and  
$\big((i x)/(\lambda^2-1)  +C\big)(g_\lambda/w)\notin  L^{2,r}$ 
for every choice of $C\in\C$, it follows that   $h\notin L^{2,r}$.
\end{proof}


According to Propositions \ref{p3}, \ref{p5} and \ref{p6}  the spectrum of $T_{p,r}$ is identified as $\sigma(T_{p,r})=\mathcal{R}_p$,    
for $1<p<\infty$  and $1\le r< \infty$. We now turn our attention
to study the fine spectra of $T_{p,r}$. Of course,
the identification of the point spectrum $\sigma_{\text{pt}}(T_{p,r}) $ is
already given in  Proposition \ref{p4}.


\begin{remark}\label{r2}
Recall, for the spectrum of a bounded linear operator $B$ on a Banach space,
that $\sigma_{\mathrm{r}}(B)\subseteq \sigma_{\mathrm{pt}}(B^*)$
as well as 
$$
\sigma_{\mathrm{pt}}(B^*)\subseteq \sigma_{\mathrm{pt}}(B)\cup \sigma_{\mathrm{r}}(B)
\quad\text{and}\quad 
\sigma_{\mathrm{pt}}(B)\subseteq \sigma_{\mathrm{pt}}(B^*)\cup \sigma_{\mathrm{r}}(B^*),
$$
\cite[Theorem 5.13]{jorgens}. In the event that $\sigma_{\mathrm{pt}}(B)=\emptyset$
it follows that $\sigma_{\mathrm{r}}(B)= \sigma_{\mathrm{pt}}(B^*)$.
\end{remark}


\begin{theorem}\label{t3}
Let  $1<p<\infty$  and $1\le r< \infty$. Then
$$
\sigma(T_{p,r})=\mathcal{R}_p.
$$

Regarding the fine spectra of $T_{p,r}$ the following 
identifications hold.
\begin{itemize}
\item [(a)]  Let $1<p<2$ and $1\le r< \infty$. Then 
$$
\sigma_{\mathrm{pt}}(T_{p,r}) =\mathrm{int}(\mathcal{R}_p);\;
\sigma_{\mathrm{r}}(T_{p,r})=\emptyset;\;
\sigma_{\mathrm{c}}(T_{p,r})=\partial\mathcal{R}_p.
$$
\item[(b)] Let $2<p<\infty$ and $1<r<\infty$. Then 
$$
\sigma_{\mathrm{pt}}(T_{p,r}) =\emptyset;\;
\sigma_{\mathrm{r}}(T_{p,r})=\mathrm{int}(\mathcal{R}_p);\;
\sigma_{\mathrm{c}}(T_{p,r})=\partial\mathcal{R}_p.
$$
\item[(c)] Let $2<p<\infty$ and $r=1$. Then 
$$
\sigma_{\mathrm{pt}}(T_{p,1}) =\emptyset;\;
\sigma_{\mathrm{r}}(T_{p,1})=\mathcal{R}_p\setminus\{\pm1\};\;
\sigma_{\mathrm{c}}(T_{p,1})=\{\pm1\}.
$$
\item [(d)]  Let $p=2$ and $r=1$. Then 
$$
\sigma_{\mathrm{pt}}(T_{2,1}) =\emptyset;\;
\sigma_{\mathrm{r}}(T_{2,1})=(-1,1);\;
\sigma_{\mathrm{c}}(T_{2,1})=\{\pm1\}.
$$
\item[(e)] Let $p=2$ and $1<r<\infty$. Then 
$$
\sigma_{\mathrm{pt}}(T_{2,r}) =\emptyset;\;
\sigma_{\mathrm{r}}(T_{2,r})=\emptyset;\;
\sigma_{\mathrm{c}}(T_{2,r})=[-1,1].
$$
\end{itemize}
\end{theorem}


\begin{proof} The fact that $\sigma(T_{p,r})=\mathcal{R}_p$ was noted
prior to Remark \ref{r2}.

(a)  To see that $\sigma_{\mathrm{r}}(T_{p,r})=\emptyset$, let $\lambda\in\C$. Since
$L^2\subseteq L^{p,r}$, it follows that $(\lambda I-T)(L^2)\subseteq (\lambda I-T)(L^{p,r})$.
Note, for $\lambda\notin\sigma(T_2)$ that $(\lambda I-T)(L^2)=L^2$ and for
$\lambda\in\sigma(T_2)$, since $\sigma(T_2)=\sigma_{\mathrm{c}}(T_2)$, that
$(\lambda I-T)(L^2)$ is dense in $L^2$. The density of $L^2$  in  $L^{p,r}$ implies
that $(\lambda I-T)(L^{p,r})$ is dense in $L^{p,r}$. So, 
$\lambda\notin\sigma_{\mathrm{r}}(T_{p,r})$. Hence, $\sigma_{\mathrm{r}}(T_{p,r})=\emptyset$.
Moreover $\sigma_{\mathrm{pt}}(T_{p,r}) =\mathrm{int}(\mathcal{R}_p)$;
see Proposition \ref{p4}(a). Accordingly, as $\sigma(T_{p,r}) =\mathcal{R}_p$, it follows that
$\sigma_{\mathrm{c}}(T_{p,r}) =\partial \mathcal{R}_p$.

(b) Since $\sigma_{\mathrm{pt}}(T_{p,r})=\emptyset$ 
(see Proposition \ref{p4}(c)), it follows
from Remark \ref{r2} and  the symmetry of the point spectrum 
in Proposition \ref{p1}(a)  that
$$
\sigma_{\mathrm{r}}(T_{p,r})=\sigma_{\mathrm{pt}}((T_{p,r})^*)=
\sigma_{\mathrm{pt}}(-T_{p',r'})=\sigma_{\mathrm{pt}}(T_{p',r'}).
$$ 
But,
$1<p'<2$ and $1<r'<\infty$. So, $\sigma_{\mathrm{pt}}(T_{p',r'})
=\mathrm{int}(\mathcal{R}_{p'}) =\mathrm{int}(\mathcal{R}_{p})$;
see Proposition \ref{p4}(a).
Hence, $\sigma_{\mathrm{r}}(T_{p,r})=\mathrm{int}(\mathcal{R}_{p})$.
Finally, $\sigma_{\mathrm{c}}(T_{p,r})=
\sigma(T_{p,r})\setminus \sigma_{\mathrm{r}}(T_{p,r})=
\mathcal{R}_{p}\setminus\mathrm{int}(\mathcal{R}_{p})
=\partial \mathcal{R}_{p}$.

(c) Arguing as in (b), since 
$\sigma_{\mathrm{pt}}(T_{p,1})=\emptyset$ by Proposition \ref{p4}(c), it again follows 
from Remark \ref{r2} that  
$\sigma_{\mathrm{r}}(T_{p,1})=\sigma_{\mathrm{pt}}((T_{p,1})^*)=
\sigma_{\mathrm{pt}}(-T_{p',\infty})=\sigma_{\mathrm{pt}}(T_{p',\infty})=
\mathcal{R}_{p'}\setminus\{\pm1\}=\mathcal{R}_{p}\setminus\{\pm1\}$;
see Remark \ref{r1}(c). 
It is now clear that $\sigma_{\mathrm{c}}(T_{p,1})=\{\pm1\}$.

(d) Since $\sigma_{\mathrm{pt}}(T_{2,1})=\emptyset$
by Proposition \ref{p4}(b), Remark \ref{r2}
implies that $\sigma_{\mathrm{r}}(T_{2,1})=\sigma_{\mathrm{pt}}((T_{2,1})^*)$ and so, by
the symmetry of the point spectrum, it follows (again by Remark \ref{r1}(c)) that  
$\sigma_{\mathrm{r}}(T_{2,1})=\sigma_{\mathrm{pt}}((T_{2,1})^*)
=\sigma_{\mathrm{pt}}(-T_{2,\infty})=\sigma_{\mathrm{pt}}(T_{2,\infty})=(-1,1)$.
Accordingly, $\sigma_{\mathrm{c}}(T_{2,1})=[-1,1]\setminus(-1,1)=\{\pm1\}$.

(e) Since $\sigma_{\mathrm{pt}}(T_{2,r})=\emptyset$,
Remark \ref{r2} implies that $\sigma_{\mathrm{r}}(T_{2,r})=\sigma_{\mathrm{pt}}((T_{2,r})^*)$.
So, we can use the symmetry of the point spectrum 
to obtain $\sigma_{\mathrm{r}}(T_{2,r})=\sigma_{\mathrm{pt}}((T_{2,r})^*)
=\sigma_{\mathrm{pt}}(-T_{2,r'})=\sigma_{\mathrm{pt}}(T_{2,r'})=\emptyset$.
The last equality is due to Proposition \ref{p4}(b) as $1<r'<\infty$.
\end{proof}


The following table gives a complete summary of the results in 
Theorem \ref{t3}. 


\medskip
\begin{center}
\begin{tabular}{ |c | c| c | c |  c| }
\hline
   	$L^{p,r}$ &$\sigma(T_{p,r})=\mathcal{R}_p$&$\sigma_{\mathrm{pt}}(T_{p,r})$& $\sigma_{\mathrm{r}}		(T_{p,r})$&$\sigma_{\mathrm{c}} (T_{p,r})$
\\ \hline 
     	&& & &
\\ \hline
 	$1<p<2$&$1\le r<\infty$  &$\mathrm{int}(\mathcal{R}_p)$ & $\emptyset$&$\partial \mathcal{R}_p$
\\ \hline
     	&&& &
\\ \hline
  	$2<p<\infty$& $r=1$ &$\emptyset$&$\mathcal{R}_p\setminus\{\pm1\}$&$\{\pm1\}$
\\ \hline
 	&$1< r<\infty$  & $\emptyset$&$\mathrm{int}(\mathcal{R}_p)$  &$\partial \mathcal{R}_p$
\\ \hline
   	&& &&
\\ \hline
   	$p=2$&$r=1$ &$\emptyset$&$(-1,1)$&$\{\pm1\}$
\\ \hline
 	&$1< r<\infty$  & $\emptyset$&$\emptyset$  &$[-1,1]$
\\ \hline
\end{tabular}
\end{center}
\medskip


\section{Further general properties of the spectrum of $T_X$}
\label{S5}


It turns out that the spectrum of $T_X$ has a number of important general properties.
We will use the fact that the functions $|x|^{-1/2}$ and $(1-x^2)^{-1/2}$ on $(-1,1)$ are 
equimeasurable  and hence,
that $L^{2,\infty}$ is the smallest r.i.\ space containing $|x|^{-1/ 2}$. That is, 
$|x|^{-1/2}\in X$ if and only if $L^{2,\infty}\subseteq X$.

\begin{proposition}\label{p7}
Let $X$ be a  separable r.i.\ space  on $(-1,1)$  such that   $0<\underline{\alpha}_X\le \overline{\alpha}_X<1$. 
Then necessarily
$$
0\in\sigma(T_X).
$$

Moreover, precisely one of the following alternatives holds.
\begin{itemize}
\item[(a)]  The following conditions are equivalent. 
\begin{itemize}
\item[(i)] $0\in\sigma_{\mathrm{pt}}(T_X)$.
\item[(ii)] $\sigma_{\mathrm{pt}}(T_X)\not=\emptyset$.
\item[(iii)] $|x|^{-1/2}\in X$.
\item[(iv)] $L^{2,\infty}\subseteq X$.
\item[(v)] $(-1,1)\subseteq\sigma_{\mathrm{pt}}(T_X)$.
\end{itemize}

\item[(b)] The following conditions are equivalent. 
\begin{itemize}
\item[(i)] $0\in\sigma_{\mathrm{r}}(T_X)$.
\item[(ii)] $\sigma_{\mathrm{r}}(T_X)\not=\emptyset$.
\item[(iii)] $|x|^{-1/2}\in X'$.
\item[(iv)] $X\subseteq L^{2,1}$.
\item[(v)] $(-1,1)\subseteq\sigma_{\mathrm{r}}(T_X)$.
\end{itemize}

\item[(c)]  The following conditions are equivalent. 
\begin{itemize}
\item[(i)] $0\in\sigma_{\mathrm{c}}(T_X)$.
\item[(ii)] $|x|^{-1/2}\notin X$ and $|x|^{-1/2}\notin X'$.
\item[(iii)] $\sigma(T_X)=\sigma_{\mathrm{c}}(T_X)\supseteq[-1,1]$.
\end{itemize}
\end{itemize}
\end{proposition}

\begin{proof}
Recall that  $w(x):=(1-x^2)^{1/2}$ for $|x|<1$.
Suppose that  $1/w\in X$. Since $T(1/w)=0$, \cite[\S4.3,(7)]{tricomi}, it follows that $T_X$
is not injective and so $0\in\sigma(T_X)$.

Suppose that $1/w\notin X$. Since $1/w\in L^1$ it follows from
\cite[\S4.3 (22) with $n=1$]{tricomi}  that
$T(ix/w(x))=\mathbf{1}$; see also \cite[Lemma 2.7]{okada-elliott} with $f=\mathbf{1}$. 
Moreover, if $T(f)(t)=0$ for  almost all $t\in(-1,1)$ for some 
$f\in\bigcup_{1<p<\infty}L^p$,  then $f=C/w$ for some 
constant $C\in\C$, \cite[\S4.3 (14)]{tricomi}. It follows that $T(f)=\mathbf{1}$ if and only if 
$f(x)=(ix+C)/w(x)$ for $C\in\C$. But,  $1/w\notin X$ and so $(ix+C)/w(x)\notin X$ for any $C\in\C$. Consequently,
$\mathbf{1}\notin T(X)$, and so $T_X$ is not surjective. Hence, $0\in\sigma(T_X)$.

(a) (ii)$\Rightarrow$(i) Since $\sigma_{\mathrm{pt}}(T_X)$ is $\R$-balanced,
$\sigma_{\mathrm{pt}}(T_X)\not=\emptyset$ implies that  $0\in\sigma_{\mathrm{pt}}(T_X)$.
(i)$\Rightarrow$(iii) Note that $0\in\sigma_{\mathrm{pt}}(T_X)$ implies that the eigenfunction
$\xi_0=1/w\in X$, which is equivalent to $|x|^{-1/2}\in X$.
(iii)$\Rightarrow$(iv)  The condition $|x|^{-1/2}\in X$ implies that $L^{2,\infty}\subseteq X$ by definition of the
Marcinkiewicz space $L^{2,\infty}$; \cite[\S2]{curbera-okada-ricker-1}. 
(iv)$\Rightarrow$(v) For $\lambda\in(-1,1)$ we have  (see Proposition \ref{p6}) 
that $|\xi_\lambda|=1/w\in L^{2,\infty}\subseteq X$ and  
so $\lambda\in\sigma_{\mathrm{pt}}(T_X)$.
(v)$\Rightarrow$(ii) Clear.

(b) Since  $X$ is separable, we have that $X^*=X'$. 
(i)$\Rightarrow$(ii) Clear.
(ii)$\Rightarrow$(iii) By Remark \ref{r2} we have that $\sigma_{\mathrm{r}}(T_X)\subseteq
\sigma_{\mathrm{pt}}((T_X)^*) = \sigma_{\mathrm{pt}}(-T_{X'})=
\sigma_{\mathrm{pt}}(T_{X'})$. Hence, the assumption
$\sigma_{\mathrm{r}}(T_X)\not=\emptyset$ implies that
$\sigma_{\mathrm{pt}}(T_{X'})\not=\emptyset$. This,  in turn, implies, via part (a), that
 $|x|^{-1/2}\in X'$. 
(iii)$\Rightarrow$(iv)  Via condition (iii), part (a) applied to $X'$ in place of $X$
yields $L^{2,\infty}\subseteq  X'$. Since $X$ has the Fatou property it follows
that $X=X''\subseteq (L^{2,\infty})'= L^{2,1}$. 
(iv)$\Rightarrow$(v) By (iv) we have $L^{2,\infty}\subseteq  X'$.
The monotonicity of the point spectrum (see Remark \ref{r1}(a)) and 
Remark \ref{r1}(c)
imply that $(-1,1)=\sigma_{\mathrm{pt}}(L^{2,\infty}) \subseteq \sigma_{\mathrm{pt}}(T_{X'})=
\sigma_{\mathrm{pt}}((T_{X})^*)$. Then Remark \ref{r2} yields
$(-1,1)\subseteq \sigma_{\mathrm{pt}}(T_{X}) \cup \sigma_{\mathrm{r}}(T_{X})$.  
Again by
the monotonicity of the point spectrum and Theorem \ref{t3}(d)
we have that
$\sigma_{\mathrm{pt}}(T_X) \subseteq \sigma_{\mathrm{pt}}(T_{2,1})=\emptyset$.
Accordingly,  $(-1,1)\subseteq\sigma_{\mathrm{r}}(T_{X})$, which is (v). 
Clearly (v)$\Rightarrow$(i).

(c) (i)$\Rightarrow$(ii) The condition $0\in\sigma_{\mathrm{c}}(T_X)$ implies   that 
$0\notin\sigma_{\mathrm{pt}}(T_X)$
and that $0\notin\sigma_{\mathrm{r}}(T_X)$.  Hence, parts (a) and (b) imply that $|x|^{-1/2}\notin X$ 
and $|x|^{-1/2}\notin X'$.
(ii)$\Rightarrow$(iii) Under the assumed conditions, parts (a) and (b) imply that $\sigma_{\mathrm{pt}}(T_{X})=
\sigma_{\mathrm{r}}(T_{X})=\emptyset$. Hence, $\sigma(T_X)=\sigma_{\mathrm{c}}(T_X)$.

To show that $(-1,1)\subseteq\sigma(T_X)$ fix $\lambda\in(-1,1)$. By the proof of 
Proposition \ref{p6}, $|g_\lambda|=1$ and so $g_\lambda\in X$. Choose $1<p<2$ such 
that $X\subseteq L^p$. Assuming that there exists $h\in X$ satisfying $(\lambda I-T_X)(h)=g_\lambda$,
we can argue as in the proof of the Proposition \ref{p6} to conclude that
$((ix)/(\lambda^2-1))\xi_\lambda \in X$. Now the assumption that $|x|^{-1/2}\notin X$
yields a contradiction. So, $(-1,1)\subseteq \sigma(T_{X})$ and hence,
$[-1,1]\subseteq \sigma(T_{X})$ as $\sigma(T_{X})$ is a closed set in $\C$.
Finally, (iii)$\Rightarrow$(i) is clear.
\end{proof}

\begin{remark}\label{r3}
(a) That (ii)$\Rightarrow$(i) in Proposition \ref{p7}(c),
whose proof was not required,  is of interest.
The condition $|x|^{-1/2}\notin X$ implies that $T_X$ is injective, that is,
$0\notin\sigma_{\mathrm{pt}}(T_X)$. Hence, by part (a), we have that 
$\sigma_{\mathrm{pt}}(T_X)=\emptyset$. Moreover, $|x|^{-1/2}\notin X'$ implies that $T_{X'}$ is injective,  that 
is, $0\notin\sigma_{\mathrm{pt}}(T_{X'})$. This  condition implies, via part (a), that 
$\sigma_{\mathrm{pt}}(T_{X'})=\emptyset$. The condition $\sigma_{\mathrm{pt}}(T_X)=\emptyset$ 
then implies that $\sigma_{\mathrm{r}}(T_{X})=\sigma_{\mathrm{pt}}(T_{X'})=\emptyset$;
see Remark \ref{r2}. So, 
$0\notin\sigma_{\mathrm{r}}(T_X)$. Since we know that $0\in\sigma(T_X)$, it follows 
that $0\in\sigma_{\mathrm{c}}(T_X)$.

(b) Note that the equivalences in (a) still hold for non-separable spaces.

(c) It is clear that \eqref{3cases} follows from Proposition \ref{p7}.
\end{remark}


\begin{corollary}\label{c1}
Let $X$ be a  separable r.i.\ space  on $(-1,1)$  such that   $0<\underline{\alpha}_X\le \overline{\alpha}_X<1$. 
The set $\sigma_{\mathrm{r}}(T_{X})$ is $\R$-balanced, that is, 
$\alpha\lambda\in \sigma_{\mathrm{r}}(T_{X})$ for each 
$\lambda\in \sigma_{\mathrm{r}}(T_{X})$ and every $\alpha\in\R$ satisfying $|\alpha|\le1$.
\end{corollary}

\begin{proof}
Proposition \ref{p7}(b) implies that $\sigma_{\mathrm{r}}(T_{X})=\emptyset$ 
(which is an $\R$-balanced set) whenever
$0\not\in\sigma_{\mathrm{r}}(T_{X})$.
For the case when  $0\in\sigma_{\mathrm{r}}(T_{X})$,
Proposition \ref{p7}(a) implies that  $\sigma_{\mathrm{pt}}(T_{X})=\emptyset$.
Then Remark \ref{r2}  implies that $\sigma_{\mathrm{r}}(T_{X})=\sigma_{\mathrm{pt}}(T_{X'})$, which
is an $\R$-balanced set, as established in Proposition \ref{p1}(b).
\end{proof}


\begin{corollary}\label{c2}
Let $X$ be a   r.i.\ space  on $(-1,1)$  such that   $0<\underline{\alpha}_X\le \overline{\alpha}_X<1$. 
Then $[-1,1]\subseteq \sigma(T_X)$. 
\end{corollary}

\begin{proof}
If $L^{2,\infty}\subseteq X$, then $(-1,1)\subseteq 
\sigma_{\mathrm{pt}}(T_{X})\subseteq \sigma(T_{X})$ by Remark \ref{r3}(b).
On the other hand, if $L^{2,\infty}\not\subseteq X$ it was established in the
proof of Proposition \ref{p7}(c) that $(\lambda I-T_X)$ is not surjective for every $\lambda\in(-1,1)$, that is, again $(-1,1)\subseteq \sigma(T_{X})$. 
Since $\sigma(T_{X})$ is closed, it follows, 
for each of the two cases, that $[-1,1]\subseteq\sigma(T_{X})$. 

\end{proof}


The following result of Widom exhibits a rather interesting
behavior of the finite Hilbert transform, which also has  consequences 
for the study of the spectrum of $T_X$. Its proof appears 
in a discussion concerning the continuous spectrum of $T_p$; see
p.~152 of \cite{widom}.

\begin{lemma}[Widom]\label{l3}
Let $f\in \bigcup_{1<p<\infty} L^p$. Then $T(f)+ f\not=0$ a.e.\ 
and $T(f)- f\not=0$ a.e.
\end{lemma}

\begin{corollary}\label{c3}
Let $X$ be any   r.i.\ space  on $(-1,1)$  such that   $0<\underline{\alpha}_X\le \overline{\alpha}_X<1$. 
Then $\pm1\in \sigma(T_X)$. In the event that $X$ is  separable, we have 
$\pm1\in \sigma_{\mathrm{c}}(T_X)$.
\end{corollary}

\begin{proof}
Lemma \ref{l3} shows that the operators $(\pm I)-T$ are not surjective on every r.i.\ space $X$
with non-trivial Boyd indices and so $\pm1\in\sigma(T_X)$. 
Moreover, Proposition \ref{p2} shows that $\pm1\notin\sigma_{\mathrm{pt}}(T_X)$.
Accordingly, 
$\pm1\in\sigma(T_X)\setminus \sigma_{\mathrm{pt}}(T_X)=
\sigma_{\mathrm{c}}(T_X)\cup \sigma_{\mathrm{r}}(T_X)$.

Let $X$ be separable. Choose $1<p<\infty$  such that $L^p\subseteq X$. 
Then  $((\pm) I-T)(L^p)\subseteq ((\pm) I-T)(X)$. Since $\pm1\in\sigma_{\mathrm{c}}(T_p)$,
because of Theorem \ref{t1}, it  follows that $((\pm) I-T)(L^p)$ is dense in $L^p$. Hence, since 
$L^p$ is also dense in 
$X$, we can conclude that $((\pm) I-T)(X)$  is dense in $X$. Thus, $\pm1\in \sigma_{\mathrm{c}}(T_X)$.
\end{proof}


\section{The residual spectrum of $T_X$}
\label{S6}


Proposition \ref{p7} suggests a strategy to study the spectrum of $T_X$.
We know from Proposition \ref{p7}  that  always $0\in\sigma(T_X)$.
The key point is  to decide in  which part of the spectrum 0 lies.
The case when $0\in\sigma_{\mathrm{pt}}(T_X)$ is considered in Proposition \ref{p8},
when $0\in\sigma_{\mathrm{r}}(T_X)$ is considered in Proposition \ref{p9}, and
when $0\in\sigma_{\mathrm{c}}(T_X)$ is considered in Proposition \ref{p10}.  
The conditions defining each one of the three cases are structured 
according to the indices $p_X$ and $q_X$ (with the aid of Proposition \ref{p7}).


Concerning the fine spectra we need a further index associated to a 
r.i.\ space $X$ on $(-1,1)$ with $0<\underline{\alpha}_X\le \overline{\alpha}_X<1$.
Define $q_X\in(1,\infty)$ by
\begin{equation}\label{qx}
q_X:=\sup\Big\{q\in(1,\infty): X\subseteq L^{q,1}\Big\}.
\end{equation}
Choose $1<r<\infty$ such that $L^r\subseteq X$ (cf.\ Lemma \ref{l1}). Then
$r=q_{L^r}\ge q_X$ shows that necessarily $q_X<\infty$.
The index  $q_X$ can be  attained or not, depending on the space $X$, 
a fact which will be relevant for the study of the spectrum of $T$.
Note that
\begin{equation}\label{qx2}
q_X=\sup\Big\{q\in(1,\infty): X\subseteq L^{q}\Big\}.
\end{equation}
However, the right-side of \eqref{qx2}, while  giving the value $q_X$, may be attained or not
attained in different spaces than the original definition given in \eqref{qx}.
Moreover,
$$
q_X\le p_X.
$$ 
Observe that there is no r.i.\ space $X$ on $(-1,1)$ with non-trivial Boyd indices
for which $p_X=q_X$ with both $p_X$ and $q_X$ being attained.


\begin{lemma}\label{l4} Let $X$ be any r.i.\ on $(-1,1)$ with the
indices $p_X$ and $q_X$   as given in \eqref{px} and \eqref{qx}, respectively.
\begin{itemize}
\item[(a)] The following inequalities hold:
\begin{equation*}\label{beta}
0<\underline{\alpha}_X\le \underline{\beta}_X\le
1/p_X\le 1/q_X\le 
\overline{\beta}_X\le \overline{\alpha}_X<1.
\end{equation*}
\item[(b)] For the associate space $X'$  of $X$, we have that $p_{X'}=(q_X)'$, that is,
$$
(1/p_{X'})+(1/q_X)=1.
$$ 
Moreover, $p_{X'}$ is attained if and only if $q_X$ is attained.
\end{itemize}
\end{lemma}


\begin{proof}
(a) Fix $\varepsilon>0$. Then  there exists $0<t_\varepsilon<1$ such that  
$$
\underline{\beta}_X-\varepsilon<\frac{\log M_{\varphi_X}(t)}{\log t}\le \underline{\beta}_X,\quad
0<t<t_\varepsilon,
$$
which implies that
$$
t^{\underline{\beta}_X-\varepsilon}> \sup_{0<s\le1}\frac{\varphi_X(st)}{\varphi_X(s)}
\ge t^{\underline{\beta}_X},\quad 0<t<t_\varepsilon.
$$
Setting  $s=1$ yields $t^{\underline{\beta}_X-\varepsilon}>\varphi_X(t)$ for $0<t<t_\varepsilon$,
which implies that
$$
L^{\frac{1}{\underline{\beta}_X-\varepsilon},1}\subseteq \Lambda(X)\subseteq X.
$$
So, for all $p>1/(\underline{\beta}_X-\varepsilon)$ we have that $L^p\subseteq X$. Since
$\varepsilon>0$ is arbitrary, it follows that $L^p\subseteq X$, for all $p>1/\underline{\beta}_X$. Hence, 
$ \underline{\beta}_X\le 1/p_X$; see \eqref{px2}.

Again fix $\varepsilon>0$. Now choose $t_\varepsilon>1$ such that
$$
\overline{\beta}_X\le \frac{\log M_{\varphi_X}(t)}{\log t}< \overline{\beta}_X+\varepsilon,\quad
t>t_\varepsilon,
$$
which implies that
$$
t^{\overline{\beta}_X}\le \sup_{0<s\le(1/t)}\frac{\varphi_X(st)}{\varphi_X(s)}
< t^{\overline{\beta}_X+\varepsilon},\quad t>t_\varepsilon.
$$
Selecting $s=1/t$ yields $1/\varphi_X(1/t)<t^{\overline{\beta}_X+\varepsilon}$ for $t>t_\varepsilon$.
Setting $u=1/t$ gives  
$u^{\overline{\beta}_X+\varepsilon}<\varphi_X(u)$ for $0<u<u_\varepsilon=1/t_\varepsilon$.
This implies that
$$
X\subseteq {\mathcal M}_{\varphi_X}\subseteq L^{\frac{1}{\overline{\beta}_X+\varepsilon},\infty},
$$
where ${\mathcal M}_{\varphi_X}$ denotes the largest r.i.\ space with fundamental function 
$\varphi_X$ (see  \cite[Definition II.5.7 and Proposition II.5.9]{bennett-sharpley}).
So, for all $q<1/(\overline{\beta}_X+\varepsilon)$ we have $X\subseteq L^{q,1}$. 
Thus, 
for all $q<1/\overline{\beta}_X$, it follows that $X\subseteq L^{q,1}$. Hence, 
$1/q_X\le \overline{\beta}_X$; see the definition of $q_X$ in \eqref{qx}.

The inequalities $\underline{\alpha}_X\le \underline{\beta}_X$ and  
$\overline{\beta}_X\le \overline{\alpha}_X$ are known, \cite[p.178]{bennett-sharpley}.

(b) Note that $p\in(1,\infty)$ satisfies $L^{p,\infty}\subseteq X'$ if and only if $q:=p'$ satisfies
$X=X''\subseteq L^{q,1}$.
So, if $p_{X'}\le p$, then $(p_{X'})'\ge p'$, from which $q_{X}\ge (p_{X'})'$  follows. Reciprocally,
if $q_X\ge q$, then $(q_X)'\le q'$, from which $(q_X)'\ge p_{X'}$ follows and then 
$(q_X)'\ge p_{X'}$ follows, that is, $q_X\le (p_{X'})'$.
\end{proof}


The following result treats the case when $0\in\sigma_{\mathrm{pt}}(T_X)$.
Proposition \ref{p7}(a) allows its formulation to be expressed  in terms of indices. Namely,
that $p_X<2$ or, that $p_X=2$ with $p_X$ attained.


\begin{proposition}\label{p8}
Let $X$ be a separable r.i.\ space on $(-1,1)$ such that 
$0<\underline{\alpha}_X\le \overline{\alpha}_X<1$. 
\begin{itemize}
\item [(a)] Suppose that $p_X<2$ and $p_X$ is not attained. Then 
$$
\sigma_{\mathrm{pt}}(T_X) =\mathrm{int}(\mathcal{R}_{p_X});\;
\sigma_{\mathrm{r}}(T_X)=\emptyset.
$$
\item [(b)]   Suppose that $p_X\le2$ and $p_X$ is attained. Then 
$$
\sigma_{\mathrm{pt}}(T_X)=\mathcal{R}_{p_X}\setminus\{\pm1\};\;
\sigma_{\mathrm{r}}(T_X)=\emptyset.
$$
\end{itemize}%
\end{proposition}


\begin{proof} 
The equalities in parts (a) and (b) regarding the point spectrum
are immediate from parts (c) and (b) of Proposition \ref{p2}, respectively.

Consider now the equalities regarding $\sigma_{\mathrm{r}}(T_X)$.  In each 
of the  cases (a) and (b) it follows from the respective identities 
for $\sigma_{\mathrm{pt}}(T_X)$ established in  the previous paragraph 
that  $0\in\sigma_{\mathrm{pt}}(T_X)$ and so
$0\notin\sigma_{\mathrm{r}}(T_X)$. Then Proposition \ref{p7}(b) implies, as  $X$ is separable, 
that $\sigma_{\mathrm{r}}(T_X)=\emptyset$ (for both (a) and  (b)).
\end{proof}


The next result considers the situation when $0\in\sigma_{\mathrm{r}}(T_X)$.
This time Proposition \ref{p7}(b) allows us to express this  in terms of indices. Namely,
that $q_X>2$ or, that $q_X=2$ with $q_X$  attained.


\begin{proposition}\label{p9}
Let $X$ be a separable r.i.\ space on $(-1,1)$ such that 
$0<\underline{\alpha}_X\le \overline{\alpha}_X<1$. 
\begin{itemize}
\item [(a)] Suppose that  $q_X>2$ and $q_X$ is not attained. Then 
$$
\sigma_{\mathrm{pt}}(T_X) =\emptyset;\;	
\sigma_{\mathrm{r}}(T_X)=\mathrm{int}(\mathcal{R}_{q_X}).
$$
\item [(b)]   Suppose that  $q_X\ge2$ and $q_X$ is attained. Then 
$$
\sigma_{\mathrm{pt}}(T_X)=\emptyset;\; 
\sigma_{\mathrm{r}}(T_X)=  \mathcal{R}_{q_X}\setminus\{\pm1\}.
$$
\end{itemize}
\end{proposition}


\begin{proof}
(a) Since $p_X\ge q_X$ (cf. Lemma \ref{l4}(a)), it follows that
$p_X>2$ and hence, Proposition  \ref{p2}(a) implies that 
\begin{equation}\label{eq-10}
\sigma_{\mathrm{pt}}(T_X)=\emptyset.
\end{equation}
According to the definition of $q_X$ in \eqref{qx}, the condition $q_X>2$ implies that $X\subseteq L^{2,1}$.
Hence, by Proposition \ref{p7}(b), we have that $\sigma_{\mathrm{r}}(T_X)\not=\emptyset$. 
By \eqref{eq-10} and Remark \ref{r2}, together with the symmetry of the point spectrum, the identity
$$
\sigma_{\mathrm{r}}(T_X)=\sigma_{\mathrm{pt}}((T_X)^*)=
\sigma_{\mathrm{pt}}(-T_{X'})=\sigma_{\mathrm{pt}}(T_{X'})
$$
is valid. The assumption of (a), together with $p_{X'}=(q_X)'$ (cf.\
Lemma \ref{l4}(b)), imply that $p_{X'}<2$ with $p_{X'}$ not
attained. It then follows from Proposition \ref{p8}(a) applied
to $X'$ that
$$
\sigma_{\mathrm{r}}(T_{X})=\sigma_{\mathrm{pt}}(T_{X'})=
\mathrm{int}(\mathcal{R}_{p_{X'}})=\mathrm{int}(\mathcal{R}_{(q_{X})'})
=\mathrm{int}(\mathcal{R}_{q_{X}}).
$$

(b)  If $q_X>2$, then $p_X\ge q_X$ yields $p_X>2$. On the other
hand, if $q_X=2$, then $p_X\ge q_X$ implies that $p_X>2$ or that
$p_X=2$ with $p_X$ not attained (as $p_X=q_X$ with both $p_X$, $q_X$
attained is impossible). So, Proposition \ref{p2}(a) can be applied
to each of the cases  (i.e., to $q_X\ge2$) to conclude that
\eqref{eq-10} holds, i.e.,
$\sigma_{\mathrm{pt}}(T_X)=\emptyset$.
We apply again  \eqref{eq-10}, Remark \ref{r2} and  the symmetry of
the point spectrum to obtain the identity
$$
\sigma_{\mathrm{r}}(T_X)=\sigma_{\mathrm{pt}}((T_X)^*)=
\sigma_{\mathrm{pt}}(-T_{X'})=\sigma_{\mathrm{pt}}(T_{X'})
$$
as above.  Now, the assumption of (b) gives, again via Lemma
\ref{l4}(b), that $p_{X'}=(q_X)'$ and  $p_{X'}\le 2$ with $p_{X'}$
attained. So, Proposition \ref{p8}(b) applied to $X'$ implies
that
$$
 \sigma_{\mathrm{r}}(T_{X})=\sigma_{\mathrm{pt}}(T_{X'})=\mathcal{R}_{p_{X'}}\setminus\{\pm1\}=
 \mathcal{R}_{({q_X})'}\setminus\{\pm1\}=
 \mathcal{R}_{q_X}\setminus\{\pm1\}.
$$
\end{proof}


The final result of this section treats the case when $0\in\sigma_{\mathrm{c}}(T_X)$.
Proposition \ref{p7}(c) allows its formulation to be expressed  in terms of indices. Namely,
that $p_X>2$ or, that $p_X=2$ with $p_X$ not attained, together with 
$q_{X}\le2$ or, that $q_{X}=2$ with $q_X$ not attained.


\begin{proposition}\label{p10}
Let $X$ be a separable r.i.\ space on $(-1,1)$ such that 
$0<\underline{\alpha}_X\le \overline{\alpha}_X<1$. 
Suppose that  $p_X\ge2\ge q_X$ and, for those cases when either $p_X=2$ 
or $q_X=2$ occur, that they are not attained. Then 
$$
\sigma_{\mathrm{pt}}(T_X) =\sigma_{\mathrm{r}}(T_X)=\emptyset;\; 
\sigma_{\mathrm{c}}(T_X)=\sigma(T_X). 
$$
\end{proposition}


\begin{proof}
Via Proposition \ref{p2}(a) the condition $p_X>2$ or, the condition $p_X=2$ (not attained), implies that 
$\sigma_{\mathrm{pt}}(T_X)=\emptyset$. This fact, together with
Remark \ref{r2},  then yields that 
$\sigma_{\mathrm{r}}(T_X)=\sigma_{\mathrm{pt}}((T_X)^*)=\sigma_{\mathrm{pt}}(-T_{X'})=
\sigma_{\mathrm{pt}}(T_{X'})$.
The condition $2>q_X$ or the condition $2=q_X$  (not attained)
implies that  $p_{X'}>2$ or $p_{X'}=2$ with $p_{X'}$  not attained (cf. Lemma \ref{l4}(b)), from which it follows that 
$\sigma_{\mathrm{pt}}(T_{X'})=\emptyset$; see Proposition \ref{p2}(a) for $X'$. Accordingly,
$\sigma_{\mathrm{r}}(T_X)=\emptyset$. 
Clearly this fact, together with $\sigma_{\mathrm{pt}}(T_X)=\emptyset$, yields that
$\sigma_{\mathrm{c}}(T_X)=\sigma(T_X)$.
\end{proof}

The converse statement of Proposition \ref{p10} also holds, that is, if $\sigma(T_X)=\sigma_{\mathrm{c}}(T_X)$,
then the assumptions of Proposition \ref{p10} necessarily hold.


\section{The spectrum and fine spectra of $T_X$}
\label{S7}


As alluded in the Introduction, the identification of the continuous spectrum of $T_X$ (and hence, of the full spectrum)  encounters
serious difficulties.

The next result identifies a superset of $\sigma(T_{X})$ which, 
in the case when   the Boyd indices coincide,
will allow us to give a full description of the spectrum and the fine spectra of $T_X$.
Observe that the union of two sets of the form $\mathcal{R}_s$, for $1<s<\infty$, is 
again a set of the same form (obviously, the larger one).


\begin{proposition}\label{p11}
Let $X$ be a separable r.i.\ space on $(-1,1)$ such that $0<\underline{\alpha}_X\le \overline{\alpha}_X<1$.
Then
$$
\sigma(T_{X})\subseteq 
\mathcal{R}_{1/\underline{\alpha}_X}\cup \mathcal{R}_{1/\overline{\alpha}_X}.
$$
\end{proposition}


\begin{proof}
We   follow the interpolation procedure used in the proof of Proposition \ref{p3}.
Suppose that $\lambda\notin\mathcal{R}_{1/\underline{\alpha}_X}\cup \mathcal{R}_{1/\overline{\alpha}_X}$.

Since $\lambda\notin\mathcal{R}_{1/\underline{\alpha}_X}$, there exists $p$ with 
$1/\underline{\alpha}_X<p<\infty$ such that $\lambda\notin\mathcal{R}_p$. Lemma \ref{l4}(a)  
implies that $\underline{\alpha}_X\le 1/p_X$ (i.e., $p_X\le 1/\underline{\alpha}_X$) and so  $p_X<p$.
According to the  definition of $p_X$ in \eqref{px},  it follows that $L^p\subseteq L^{p,\infty}\subseteq X$.
In a similar way,
since $\lambda\notin\mathcal{R}_{1/\overline{\alpha}_X}$, there exists $q$ with 
$q<1/\overline{\alpha}_X$ such that $\lambda\notin\mathcal{R}_q$. From Lemma \ref{l4}(a) we have 
that $1/\overline{\alpha}_X\le q_X$ which yields $q<q_X$.
Via the definition of $q_X$ in \eqref{qx}, we can conclude that $X\subseteq L^{q,1}\subseteq L^q $.
So,  $L^p\subseteq  X\subseteq L^q$ with  $1/p<\underline{\alpha}_X\le
\overline{\alpha}_X<1/q$.

Noting that $\lambda\notin\mathcal{R}_p$ and  $\lambda\notin\mathcal{R}_q$, Theorem \ref{t1}
implies that  the resolvent operators 
$R(\lambda,T_q):=(\lambda I -T)^{-1}\colon L^q\to L^q$
and $R(\lambda,T_{p}):=(\lambda I -T)^{-1}\colon L^{p}\to L^{p}$
are continuous bijections.
The condition $L^p\subseteq  X\subseteq L^q$ with $1/p<\underline{\alpha}_X\le
\overline{\alpha}_X<1/q$ allows us to apply 
Boyd's interpolation theorem, \cite[Theorem 2.b.11]{lindenstrauss-tzafriri},
to conclude 
that $(\lambda I -T)^{-1}\colon X\to X$ is continuous. Set $S:=(\lambda I-T)^{-1}$. 
Then $S(\lambda I -T_X)=I$ on $L^{p}$ and $(\lambda I -T_X)S=I$ on $L^{p}$.
Since $L^{p}$ is dense in $X$, it follows that $S=(\lambda I -T_X)^{-1}$ on 
$X$. Hence, $\lambda\notin\sigma(T_X)$. 
Consequently, $\sigma(T_X)\subseteq
\mathcal{R}_{1/\underline{\alpha}_X}\cup 
\mathcal{R}_{1/\overline{\alpha}_X}$.
\end{proof}


Due to Proposition \ref{p11}, in the case when 
the lower and upper   Boyd indices of $X$ coincide, we can identify the continuous spectrum of
$T_X$ and hence, also identify the full and the fine spectra of $T_X$. Recall
that $\mathcal{R}_2=[-1,1]$.

\begin{theorem}\label{t4}
Let $X$ be a separable r.i.\ space on $(-1,1)$ 
with non-trivial Boyd indices such that $\underline{\alpha}_X= \overline{\alpha}_X$. Then
\begin{equation}\label{pqx}
p_X=q_X
=1/\underline{\alpha}_X= 1/\overline{\alpha}_X,
\end{equation}
and 
$$
\sigma(T_X)=\mathcal{R}_{p_X}.
$$

Moreover, regarding the fine spectra of $T_X$  we have the following descriptions.
\begin{itemize}
\item [(a)] Let $p_X<2$. 
\begin{itemize}
\item [(a1)] Suppose that $p_X$ is not attained. Then 
$$
\sigma_{\mathrm{pt}}(T_X) =\mathrm{int}(\mathcal{R}_{p_X});\;
\sigma_{\mathrm{r}}(T_X)=\emptyset;\;
\sigma_{\mathrm{c}}(T_X)=\partial\mathcal{R}_{p_X}.
$$
\item [(a2)]   Suppose that $p_X$ is attained. Then 
$$
\sigma_{\mathrm{pt}}(T_X)=\mathcal{R}_{p_X}\setminus\{\pm1\};\;
\sigma_{\mathrm{r}}(T_X)=\emptyset;\;
\sigma_{\mathrm{c}}(T_X)=\{\pm1\}.
$$
\end{itemize}
\item [(b)] Let  $p_X>2$. 
\begin{itemize}
\item [(b1)] Suppose that  $q_X$ is not attained. Then 
$$
\sigma_{\mathrm{pt}}(T_X) =\emptyset;\;	
\sigma_{\mathrm{r}}(T_X)=\mathrm{int}(\mathcal{R}_{p_X});\;
\sigma_{\mathrm{c}}(T_X)=\partial\mathcal{R}_{p_X}.
$$
\item [(b2)]   Suppose that  $q_X$ is attained. Then 
$$
\sigma_{\mathrm{pt}}(T_X)=\emptyset;\; 
\sigma_{\mathrm{r}}(T_X)=  \mathcal{R}_{p_X}\setminus\{\pm1\};\;
\sigma_{\mathrm{c}}(T_X)=
\{\pm1\}.
$$
\end{itemize}
\item [(c)] Let  $p_X=2$. 
\begin{itemize}
\item [(c1)]   Suppose that  $p_X$ is attained. Then 
$$
\sigma_{\mathrm{pt}}(T_X)=(-1,1);\;
\sigma_{\mathrm{r}}(T_X)=\emptyset;\;
\sigma_{\mathrm{c}}(T_X)=\{\pm1\}.
$$
\item [(c2)]   Suppose that  $q_X$ is attained. Then 
$$
\sigma_{\mathrm{pt}}(T_X)=\emptyset;\; 
\sigma_{\mathrm{r}}(T_X)=  (-1,1);\;
\sigma_{\mathrm{c}}(T_X)=
\{\pm1\}.
$$
\item [(c3)]   Suppose that  both $p_X$ and $q_X$ are not attained. Then 
$$
\sigma_{\mathrm{pt}}(T_X) =\sigma_{\mathrm{r}}(T_X)=\emptyset;\;
\sigma_{\mathrm{c}}(T_X)=\mathcal{R}_2=[-1,1].
$$
\end{itemize}
\end{itemize}
\end{theorem}


\begin{proof}
The equality of the indices in \eqref{pqx} follows from Lemma \ref{l4}(a) and 
$\underline{\alpha}_X= \overline{\alpha}_X$. Proposition \ref{p11} then implies 
that 
\begin{equation}\label{eq-20}
\sigma(T_X)\subseteq  \mathcal{R}_{p_X}.
\end{equation}

Recall that Propositions \ref{p8}, \ref{p9} and \ref{p10} identify, for each of the cases (a), (b) and (c), the 
point spectrum and the residual spectrum of $T_X$. Moreover, in cases (a), (b), (c1) and (c2) 
Propositions \ref{p8} and \ref{p9} ensure that 
$\sigma_{\mathrm{pt}}(T_X) \cup\sigma_{\mathrm{r}}(T_X)$
is dense in $\mathcal{R}_{p_X}$. Hence, it follows from \eqref{eq-20}
that 
$$
\sigma_{\mathrm{c}}(T_X) =\mathcal{R}_{p_X}\setminus 
(\sigma_{\mathrm{pt}}(T_X) \cup\sigma_{\mathrm{r}}(T_X)),
$$
and the results in (a), (b), (c1) and (c2) follow.

For the case (c3) we have $\sigma_{\mathrm{pt}}(T_X) =\sigma_{\mathrm{r}}(T_X)=\emptyset$ from
Proposition \ref{p10}. Regarding the continuous spectrum and the full spectrum, 
note that $\sigma(T_X)\subseteq  \mathcal{R}_2=[-1,1]$ because of \eqref{eq-20} 
and that $\mathcal{R}_2=[-1,1]\subseteq \sigma(T_X)$ via Corollary \ref{c2}.
\end{proof}


The following table provides an overview of the results in  Theorem \ref{t4}, where the issue
of whether  $p_X$ and  $q_X$ are attained or not attained is indicated by a.\  or n.a., respectively.


\bigskip
\begin{center}
\begin{tabular}{ |c | c| c | c |  c| }
\hline
	$\sigma(T_X)=\mathcal{R}_{p_X}$  & a./n.a. &
	$\sigma_{\mathrm{pt}}(T_X)$& $\sigma_{\mathrm{r}}(T_X)$&$\sigma_{\mathrm{c}} (T_X)$
\\ \hline 
   	&&& &
\\ \hline
	$p_X<2$ &n.a. &$\mathrm{int}(\mathcal{R}_{p_X})$ & 
	$\emptyset$ &$\partial\mathcal{R}_{p_X}	$
\\ \hline
   	& a.  & $\mathcal{R}_{p_X}\setminus\{\pm1\}$&$\emptyset$  
	&$\{\pm1\}$
\\ \hline
  	$p_X>2$ & n.a. &$\emptyset$&$\mathrm{int}(\mathcal{R}_{p_X})$ &
  	$\partial\mathcal{R}_{p_X}$
\\ \hline
	&a. &$\emptyset$&$\mathcal{R}_{p_X}\setminus\{\pm1\}$ &
  	$\{\pm1\}$
\\ \hline
	$p_X=2$ & $p_X$ a.  &$(-1,1)$& $\emptyset$& $\{\pm1\}$
\\ \hline
	& $q_X$ a. &$\emptyset$& $(-1,1)$& $\{\pm1\}$
\\ \hline
	& $p_X,q_X$ n.a.  &$\emptyset$& $\emptyset$& $\mathcal{R}_2=[-1,1]$
   	\\ \hline
		
\end{tabular}
\end{center}
\bigskip


The results shown above for the fine spectra  are in agreement  with those obtained
in Section \ref{S4}  for the Lorentz $L^{p,r}$ spaces. We exhibit below further classes of r.i.\
spaces $X$ whose lower and  upper Boyd indices coincide. Hence, for such
spaces $X$, Theorem \ref{t4}
applies and renders a full description  of the spectrum and the fine spectra of $T_X$.


\begin{example}\label{talenti}
Consider the Orlicz space $L^\Phi(-1,1)$ with $\Phi(t):=t^p \exp (\sqrt{1+s\log_{+}t})$ 
for $p>1$ and $s>0$, first
studied by Talenti, \cite{talenti}. For this space the lower  and upper Boyd 
indices coincide and equal $1/p$, \cite[Example 5.8(2)]{fiorenza-krbec}. Thus,  
Theorem \ref{t4} gives a full description of the spectra of $T$ 
acting on $L^\Phi(-1,1)$.

This example belongs to the family of all separable Orlicz spaces $L^\Phi(-1,1)$  whose 
Young function $\Phi$ and its complementary function both satisfy the $\Delta_2$-condition,
and $\Phi$ satisfies the condition
$$
\lim_{t\to0^+}\frac{t\Phi'(t)}{\Phi(t)}=\lim_{t\to+\infty}\frac{t\Phi'(t)}{\Phi(t)},
$$
which implies the equality of the Boyd indices, \cite[Theorem 1.3]{fiorenza-krbec}. 
Again, via Theorem \ref{t4},  a full description of the spectra of $T$ 
acting on $L^\Phi(-1,1)$  is available.
\end{example}

\begin{example}\label{lorentz}
Consider the classical Lorentz space $\Lambda^p(w)$ on $(-1,1)$, for $1\le p<\infty$ with 
$w$ a positive, decreasing and continuous
function on $(0,2)$ satisfying $\lim_{t\to0}w(t)=\infty$. It consists of all functions $f\colon(-1,1)\to\C$
satisfying
$$
\|f\|_{\Lambda^p(w)}:=\left(\int_0^2f^*(t)^pw(t)\,dt\right)^{1/p}<\infty.
$$
The  separable spaces $X=\Lambda^p(w)$ are of fundamental type, that is, 
$\underline{\alpha}_X=\underline{\beta}_X$ and $\overline{\alpha}_X=\overline{\beta}_X$;
see \cite[Theorem 3.1]{feher}.
Hence, the upper and lower Boyd indices are, respectively,
$$
\overline{\alpha}_X=\lim_{t\to\infty}\frac{\overline{W}(t)^{1/p}}{\log t},
\quad
\underline{\alpha}_X=\lim_{t\to0}\frac{\overline{W}(t)^{1/p}}{\log t},
$$
where $W(t):=\int_0^t w(x)\,dx$ and $\overline{W}(t):=\sup_{0<st<2} W(ts)/W(s)$. 
Consider, for example, the weight $w(t)=t^{a-1}\varphi(t)$ for $0<a<1$ and 
with $\varphi$ 
a slowly varying function. Then  the lower  and upper  Boyd indices  of $\Lambda^p(w)$ 
coincide and equal $a$ (see the proof of \cite[Example 2.3]{nazeer-etal}).
So, Theorem \ref{t4} gives a full description of the spectra of $T$ 
acting on $\Lambda^p(w)$.
\end{example}

\begin{example}\label{grand}
The grand Lebesgue space $L^{p)}$, for $1<p<\infty$, was introduced by Iwaniec and Sbordone
in connection with the integrability of the Jacobian, \cite{iwaniec-sbordone}.
It consists of all measurable functions $f$ on $(0,1)$ such that
$$
\sup_{0<\varepsilon<p-1}\left(\varepsilon\int_0^1 |f(x)|^{p-\varepsilon}\,dx
\right)^{1/(p-\varepsilon)}<\infty.
$$
It is a non-separable  r.i.\ space with fundamental function 
$t\to t^{1/p}\log^{-1/p}(1/t)$,  \cite[Theorem 3.7]{fiorenza}. Its associate space is the so-called small 
Lebesgue space $L^{p)'}$ defined in \cite{fiorenza}, which is a separable, non-reflexive r.i.\ space.
These spaces have found a wealth of applications. They can also be defined 
over $(-1,1)$. In \cite[Theorem 2.1]{formica-giova} it is proved, for
$L^{p)}$ and hence, also holds for $L^{p)'}$, that the Boyd indices coincide and equal $1/p$
(in  \cite{formica-giova} a generalized version of these spaces is considered). Hence, again
Theorem \ref{t4} is applicable and provides the spectrum 
and fine spectra of $T$ acting on $L^{p)'}$.
\end{example}


The proof of Theorem \ref{t4} relied on finding an appropriate superset of $\sigma(T_{X})$, via Proposition \ref{p11} 
(whose proof relies on  Boyd's interpolation theorem). It is also possible to find a smaller superset of $\sigma(T_{X})$ 
by using the indices $p_X$ and $q_X$ (in place of $\underline{\alpha}_X, \overline{\alpha}_X$)
and applying a different interpolation theorem. This superset can then be used
to obtain a full description of the spectrum and the fine spectra of $T_X$.


\begin{proposition}\label{p12}
Let $X$ be a separable r.i.\ space on $(-1,1)$ such that $0<\underline{\alpha}_X\le \overline{\alpha}_X<1$.
Suppose that $2>p_X=q_X$ and $q_X$ is attained. If $X$ is an interpolation space between
$L^2$ and $L^{p_X}$, then
$$
\sigma(T_{X})\subseteq 
\mathcal{R}_{p_X}.
$$
\end{proposition}


\begin{proof}
Suppose that $\lambda\notin\mathcal{R}_{p_X}$. Since 
$\mathcal{R}_{p_X}=\mathcal{R}_{q_X}\supseteq\mathcal{R}_2$, it follows that 
$\lambda\notin\mathcal{R}_{q_X}$ and $\lambda\notin\mathcal{R}_2$.
Then Theorem \ref{t1}  implies that  the resolvent operators 
$R(\lambda,T_{q_X}):=(\lambda I -T)^{-1}\colon L^{q_X}\to L^{q_X}$
and $R(\lambda,T_{2}):=(\lambda I -T)^{-1}\colon L^{2}\to L^{2}$
are continuous bijections.
Since $2>p_X=q_X$ with $q_X$  attained, according to the  definition of $p_X$ in \eqref{px}
and to the definition of $q_X$ in \eqref{qx} with  $q_X$
attained, we can conclude that $L^2\subseteq X\subseteq L^{q_X}$.
By the assumptions on  $X$ it is an interpolation space between
$L^2$ and $L^{q_X}$, from which it follows that $(\lambda I -T)^{-1}\colon X\to X$ is continuous. 
Set $S:=(\lambda I-T)^{-1}$. 
Then $S\circ(\lambda I -T_X)=I$ on $L^{2}$ and $(\lambda I -T_X)\circ S=I$ on $L^{2}$.
Since $L^{2}$ is dense in $X$, it follows that $S=(\lambda I -T_X)^{-1}$ on 
$X$. Hence, $\lambda\notin\sigma(T_X)$. 
\end{proof}


Under the conditions of Proposition \ref{p12} we can identify the continuous spectrum of
$T_X$ and hence, also the full spectrum and the fine spectra of $T_X$.

\begin{theorem}\label{t5}
Let $X$ be a separable r.i.\ space on $(-1,1)$ such that $0<\underline{\alpha}_X\le \overline{\alpha}_X<1$.
Suppose that $2>p_X=q_X$ and $q_X$ is attained. If  $X$ is 
an interpolation space between $L^2$ and $L^{p_X}$, then
$$
\sigma(T_X)=\mathcal{R}_{p_X}.
$$

Moreover, regarding the fine spectra of $T_X$  we have the following descriptions.
$$
\sigma_{\mathrm{pt}}(T_X) =\mathrm{int}(\mathcal{R}_{p_X});\;
\sigma_{\mathrm{r}}(T_X)=\emptyset;\;
\sigma_{\mathrm{c}}(T_X)=\partial\mathcal{R}_{p_X}.
$$
\end{theorem}


\begin{proof}
Since $q_X$ is attained and $p_X=q_X$, the index $p_X$ is not attained; see Section \ref{S6}.
Since $2>p_X$  and $p_X$ is not attained, Proposition \ref{p8}(a) yields that 
$\sigma_{\mathrm{pt}}(T_X) =\mathrm{int}(\mathcal{R}_{p_X})$ and 
$\sigma_{\mathrm{r}}(T_X)=\emptyset$. But, $\sigma(T_X)\subseteq\mathcal{R}_{p_X}$ (cf. 
Proposition \ref{p12}), and so $\sigma(T_X)=\mathcal{R}_{p_X}$. Consequently,
$\sigma_{\mathrm{c}}(T_X)=\partial\mathcal{R}_{p_X}$,
\end{proof}


The following symmetrical analogues of Proposition \ref{p12} and Theorem \ref{t5} are valid, with  
a similar proof (this time using Proposition \ref{p9}(a)).

\begin{proposition}\label{p13}
Let $X$ be a separable r.i.\ space on $(-1,1)$ such that $0<\underline{\alpha}_X\le \overline{\alpha}_X<1$.
Suppose that $p_X=q_X>2$ and $p_X$ is attained. If  $X$ is an interpolation space between
$L^2$ and $L^{p_X}$, then 
$$
\sigma(T_{X})\subseteq 
\mathcal{R}_{p_X}.
$$
\end{proposition}

\begin{theorem}\label{t6}
Let $X$ be a separable r.i.\ space on $(-1,1)$ such that $0<\underline{\alpha}_X\le \overline{\alpha}_X<1$.
Suppose that $p_X=q_X>2$ and $p_X$ is attained. If $X$ is an interpolation space 
between $L^2$ and $L^{p_X}$, then
$$
\sigma(T_X)=\mathcal{R}_{p_X}.
$$

Moreover, regarding the fine spectra of $T_X$  we have the following descriptions.
$$
\sigma_{\mathrm{pt}}(T_X) =\emptyset;\;
\sigma_{\mathrm{r}}(T_X)=\mathrm{int}(\mathcal{R}_{p_X});\;
\sigma_{\mathrm{c}}(T_X)=\partial\mathcal{R}_{p_X}.
$$
\end{theorem}


\begin{remark}\label{r4}  
Interpolation spaces between $L^p$ and $L^q$, for $1<p<q<\infty$, 
have a rather transparent description as spaces which 
are both interpolation spaces between $L^1$ and $L^q$ and
interpolation spaces between $L^p$ and $L^\infty$; see \cite{arazy-cwikel} and
\cite{lorentz-shimogaki}. 
\end{remark}


\begin{remark}\label{r5}
The case of non-separable r.i.\ spaces exhibits further difficulties.
We exemplify this via the non-separable Lorentz $L^{p,\infty}$ spaces  for
$1<p<\infty$. 
The spectrum is $\sigma(T_{p,\infty}) =\mathcal{R}_p $. To see this,
consider  $1<p'<\infty$. An appeal to   Propositions \ref{p3}, 
\ref{p5} and \ref{p6} and the symmetry of the spectrum, yields that 
$\mathcal{R}_p=\mathcal{R}_{p'}=
\sigma(T_{p',1})=\sigma((T_{p',1})^*)=\sigma(-T_{p,\infty})=\sigma(T_{p,\infty})$.

Regarding the point spectrum,     Proposition \ref{p2} is relevant. Note that $p_{L^{p,\infty}}=p$
is attained. For $1<p<2$  we have that 
$\sigma_{\mathrm{pt}}(T_{p,\infty})=\mathcal{R}_p\setminus \{\pm1\}$,  for $p=2$ that
$\sigma_{\mathrm{pt}}(T_{2,\infty})=(-1,1)$, and 
for $2<p<\infty$  that $\sigma_{\mathrm{pt}}(T_{p,\infty})=\emptyset$.
Regarding the residual and continuous  spectra, consider the case when $2<p<\infty$. 
Then $1<p'<2$ with $(T_{p',1})^*=- T_{p,\infty}$ and $\sigma_{\mathrm{pt}}(T_{p,\infty})=\emptyset$. 
Via Remark \ref{r2} and Proposition \ref{p4}(a) it follows
that $\mathrm{int}(\mathcal{R}_p)=\mathrm{int}(\mathcal{R}_{p'})
=\sigma_{\mathrm{pt}}(T_{p',1})\subseteq  \sigma_{\mathrm{r}}(- T_{p,\infty})
=\sigma_{\mathrm{r}}(T_{p,\infty})$, due to symmetry of the residual spectrum.
Thus, $\mathrm{int}(\mathcal{R}_p)\subseteq  \sigma_{\mathrm{r}}(T_{p,\infty})$ which implies that  
$\sigma_{\mathrm{c}}(T_{p,\infty})\subseteq \partial \mathcal{R}_p$.
In the case when $1<p\le 2$, we have that 
$\sigma_{\mathrm{r}}(T_{p,\infty})\cup \sigma_{\mathrm{c}}(T_{p,\infty})=\{\pm1\}$,
which is inconclusive. 
\end{remark}



\end{document}